\documentclass[final]{siamart}

\usepackage{lipsum}
\usepackage{amsfonts}
\usepackage{graphicx}
\usepackage{epstopdf}
\usepackage{algorithmic}
\ifpdf
  \DeclareGraphicsExtensions{.eps,.pdf,.png,.jpg}
\else
  \DeclareGraphicsExtensions{.eps}
\fi

\newcommand{\TheTitle}{Analysis of the Ensemble Kalman Filter for Inverse Problems} 
\newcommand{\TheAuthors}{C. Schillings and A. M. Stuart}

\headers{\TheTitle}{\TheAuthors}

\title{{\TheTitle}}%\thanks{This work was funded by the Fog Research Institute under contract no.~FRI-454.}}

\author{
  Claudia Schillings\thanks{Mathematics Institute, University of Warwick, Coventry CV4 7AL, UK 
    (\email{c.schillings@warwick.ac.uk)}.}
  \and
  Andrew M. Stuart\thanks{Mathematics Institute, University of Warwick, Coventry CV4 7AL, UK (\email{a.m.stuart@warwick.ac.uk}).}
}

\usepackage{amsopn}

\ifpdf
\hypersetup{
  pdftitle={\TheTitle},
  pdfauthor={\TheAuthors}
}
\fi

\usepackage{mathtools}
\theoremstyle{plain}

\theoremstyle{remark}

\usepackage{autonum}

\newcommand{\cX}{\mathcal X}
\newcommand{\cY}{\mathcal Y}
\newcommand{\bbR}{\mathbb R}
\newcommand{\dd}{{\mathrm d}}
\newcommand{\cG}{\mathcal{G}}
\newcommand{\bu}{\overline{u}}
\newcommand{\bG}{\overline{\cG}}

\newcommand{\ud}{u^{\dagger}}
\newcommand{\one}{\mathsf l}

\newcommand{\Xp}{\cX^{\perp}}
\newcommand{\Yl}{\cY^{\|}}
\newcommand{\Yp}{\cY^{\perp}}
\newcommand{\rl}{r^{(j)}_{\|}}
\newcommand{\rp}{r^{(j)}_{\perp}}

\newcommand{\dll}{\vartheta^{(j)}_{\|}}
\newcommand{\dpp}{\vartheta^{(j)}_{\perp}}

\definecolor{darkred}{rgb}{.7,0,0}

\newcommand{\beq}{\begin{equation}}
\newcommand{\eeq}{\end{equation}}
\newcommand{\bea}{\begin{eqnarray}}
\newcommand{\eea}{\end{eqnarray}}
\newcommand{\beas}{\begin{eqnarray*}}
\newcommand{\eeas}{\end{eqnarray*}}

\def\b1{{\bf 1}}

\def\cG{{\mathcal G}}

\newcommand{\rd}{\mathrm{d}}

\newcommand{\cls}[1]{{\color{black}{#1}}}

\newcommand{\diverg}{\mbox{div}}

\newcommand{\spann}{\operatornamewithlimits{span}}
\newcommand{\id}{\operatornamewithlimits{id}}

\begin{document}

\maketitle

% REQUIRED
\begin{abstract}
  The ensemble Kalman filter (EnKF) is a widely used methodology for
state estimation in partial, noisily observed dynamical systems,
and for parameter estimation in inverse problems. Despite its
widespread use in the geophysical sciences, and its gradual adoption
in many other areas of application, analysis of the method is in its infancy.
Furthermore, much of the existing analysis deals with the large ensemble
limit, far from the regime in which the method is typically used.
The goal of this paper is to analyze the method when applied to
inverse problems with fixed ensemble size. A continuous-time limit 
is derived and the long-time behavior of the resulting dynamical 
system is studied.  Most of the rigorous analysis is confined to 
the \cls{linear forward problem}, where we demonstrate that the
continuous time limit of the EnKF corresponds to a set of gradient
flows for the data misfit in each ensemble member, coupled through
a common pre-conditioner which is the empirical covariance matrix 
of the ensemble.  Numerical results demonstrate that the conclusions
of the analysis extend beyond the linear inverse problem setting.  
Numerical experiments are also given which demonstrate the benefits
of various extensions of the basic methodology.
\end{abstract}

% REQUIRED

% REQUIRED

\section{Introduction}
\label{sec:I}

The Ensemble Kalman filter (EnKF) has had enormous impact on the
applied sciences since its introduction in the 1990s by Evensen
and coworkers; see \cite{evensen2003ensemble} for an overview. It
is used for both data assimilation problems, where the objective
is to estimate a partially observed time-evolving system 
\cls{sequentially in time} \cite{kalnay2003atmospheric}, and inverse problems, where the
objective is to estimate a (typically distributed) parameter appearing
in a differential equation \cite{oliver2008inverse}. Much of the
analysis of the method has focussed on the large ensemble limit 
 \cls{\cite{li2008numerical,le2009large, gratton2014convergence, kwiatkowski2015convergence, ernst2015analysis,law2016deterministic}}. However the primary reason for the 
adoption of the
method by practitioners is its robustness and perceived effectiveness
when used with small ensemble sizes, as discussed
in \cite{bergemann2010localization,bergemann2010mollified} for example.
It is therefore important to
study the properties of the EnKF for fixed ensemble size, in order
to better understand current practice, and to suggest future directions
for development of the method. Such fixed ensemble size
analyses are starting to appear in the literature for
both data assimilation problems \cite{kelly2014well,tongnonlinear}
and inverse problems \cite{iglesias2013ensemble, iglesias2014iterative, iglesias2015regularizing}. 
In this paper
we analyze the EnKF for inverse problems, adding greater depth 
to our understanding of the basic method, as formulated in 
\cite{iglesias2013ensemble}, as well as variants on the basic method
which employ techniques such as variance inflation and localization 
(see \cite{KLS} and the references therein), together with new
ideas (introduced here) which borrow from the use of sequential 
Monte Carlo (SMC) method for inverse problems introduced in
\cite{kantas2014sequential}.

Let $\cG: \cX \to \cY$ \cls{be a continuous mapping between separable Hilbert spaces} 
$\cX$ and $\cY$.
We are interested in the inverse problem of recovering unknown $u$ from 
observation $y$ where
\begin{equation}
\label{eq:ip}
y=\cG(u)+\eta.
\end{equation}
Here $\eta$ is an observational noise. \cls{We are
typically interested in the case where the inversion is ill-posed on $\cY$, i.e.\ one of the following three conditions is violated: existence of solutions, uniqueness, stability. In the linear setting, we can think for example of a compact operator violating the stability condition.} Indeed 
throughout we assume in all our rigorous results, without comment,
that $\cY=\mathbb R^K$ for $K\in \mathbb N$ except
in a few particular places where we explicitly state that $\cY$
is infinite dimensional. 
A key role in such inverse problems is played by
the least squares functional 
\begin{equation}
\label{eq:lsq}
\Phi(u;y)=\frac12\|\Gamma^{-\frac12}(y-\cG(u))\|_{\cY}^2.
\end{equation}
Here $\Gamma>0$ normalizes the model-data misfit and often knowledge
of the covariance structure of typical noise $\eta$ is used to define $\Gamma$.

When the inverse problem is ill-posed, infimization of $\Phi$
in $\cX$ is not a well-behaved problem and some form of regularization
is required. Classical methods include Tikhonov regularization,
infimization over a compact ball in $\cX$ and truncated iterative methods
\cite{engl1996regularization}. An alternative approach is Bayesian
regularization. In Bayesian regularization
$(u,y)$ is viewed as a jointly varying random variable
in $\cX \times \cY$ and, assuming that $\eta \sim N(0,\Gamma)$ is
independent of $u \sim \mu_0$, the solution to the inverse problem
is the $\cX-$valued random variable $u|y$ distributed according to
measure
\begin{equation}
\label{eq:bip}
\mu(du)=\frac{1}{Z}\exp\bigl(-\Phi(u;y)\bigr)\mu_0(du),
\end{equation} 
where $Z$ is chosen so that $\mu$ is a probability measure:
$$Z:=\int_{\cX} \exp\bigl(-\Phi(u;y)\bigr)\mu_0(du).$$
See \cite{DashtiStuart} for details concerning the Bayesian
methodology. 

The EnKF is derived within the Bayesian framework and, through its
ensemble properties, is viewed as approximating the posterior
distribution on the random variable $u|y.$ However, except in the
large sample limit \cls{for linear problems} 
\cite{le2009large, gratton2014convergence, kwiatkowski2015convergence} 
there is little to substantiate this viewpoint; indeed the paper 
\cite{ernst2015analysis} demonstrates this quite clearly by showing that 
for nonlinear problems the large ensemble limit does not
approximate the posterior distribution. \cls{In \cite{law2016deterministic}, 
a related result is proved for the EnKF in the context of data assimilaion;
in the large ensemble size limit the EnKF is proved to converge to the mean-field 
EnKF, which provides the optimal linear estimator of the conditional mean, but
does not reproduce the filtering distribution, except in the linear Gaussian
case.} A different perspective on
the EnKF is that it constitutes a derivative-free optimization technique,
with the ensemble used as a proxy for derivative information.
This optimization viewpoint was adopted in 
\cite{iglesias2013ensemble,iglesias2014iterative} 
and is the one we take in this paper: through analysis and numerical
experiments we study the properties of the EnKF as a 
regularization technique for minimization of the least-squares misfit
functional $\Phi$ \cls{at fixed ensemble size.} 
We do, however, use the Bayesian perspective to
derive the algorithm, and to suggest variants of it.

{In section \ref{sec:V} we describe the EnKF in its basic form,
deriving the algorithm by means of ideas from 
SMC applied to the Bayesian inverse problem, together with invocation
of a Gaussian approximation.
Section \ref{sec:C} describes continuous time limits of
the method, leading to differential equations, and in
section \ref{sec:asy} we study properties of the
differential equations derived in the linear case and, in particular, 
their long-time behavior. Using this analysis we obtain clear
understanding of the sense in which the EnKF is a regularized
optimization method for $\Phi$. 
Indeed we show in section \ref{sec:C}
that the continuous time limit of the EnKF 
corresponds to a set of preconditioned gradient flows for 
the data misfit in each ensemble member. The 
common preconditioner is the empirical covariance matrix
of the ensemble which thereby couples the ensemble members together
and renders the algorithm nonlinear, even for linear inverse
problems. 
Section \ref{sec:N} is devoted to numerical studies which
illustrate the foregoing theory for linear problems, and which also
demonstrate that similar ideas apply to nonlinear problems.
In section \ref{sec:CV} we discuss variants of the basic EnKF method,
in particular the addition of variance inflation, localization or the use of
random search methods based on SMC, within the
ensemble method; all of these methods break the
invariant subspace property of the basic EnKF proved in
\cite{iglesias2013ensemble} and we explore
numerically the benefits of doing so.

\section{The EnKF for Inverse Problems}
\label{sec:V}

Here we describe how to derive the iterative EnKF as an
approximation of the SMC method for
inverse problems.  Recall the posterior distribution $\mu$
given by \eqref{eq:bip} and define the probability measures 
$\mu_n$ by, for $h=N^{-1}$,
\begin{equation}
\mu_{n}(du) \propto \exp\bigl(-nh\Phi(u;y)\bigr)\mu_0(du).
\end{equation} 
\cls{The measures $\mu_n$ are intermediate measures defined via likelihoods scaled by the step size $h=N^{-1}$.} 
It follows that $\mu_N=\mu$ the desired measure on $u|y.$ Then
\begin{equation}
\label{eq:bipi}
\mu_{n+1}(du)=\frac{1}{Z_n}\exp\bigl(-h\Phi(u;y)\bigr)\mu_n(du),
\end{equation} 
for 
$$Z_n=\int \exp(-h\cls{\Phi(u;y)})\mu_n(du).$$
Denoting by $L_n$ the nonlinear operator 
corresponding to application of Bayes' theorem to map from $\mu_n$
to $\mu_{n+1}$ we have 
\begin{equation}
\label{eq:smc1}
  {\mu_{n+1}=L_n\mu_n}\; .
  \end{equation}
\cls{We have introduced an articial discrete time dynamical system which
maps the prior $\mu_0$ into the posterior $\mu_N=\mu.$
A heuristic worthy of note is that although we look at the data $y$
at each of $N$ steps, the effective variance is amplified by $N=1/h$ at each
step, compensating for the redundant, repeated use of the data.}
The idea of SMC is to approximate $\mu_n$ by a weighted sum
of Dirac masses: given  a set of particles and weights
$\{u_{n}^{(j)},w_{n}^{(j)}\}_{j=1}^J$ the approximation takes
the form
$${\mu_{n}}\simeq \sum_{j=1}^J w_n^{(j)}\cls{\delta_{{u_n}^{(j)}}},$$
\cls{with $\delta_{{u_n}^{(j)}}$ denoting the delta-Dirac mass located at ${u_n}^{(j)}.$}
The method is defined by the mapping of the particles and weights
at time $n$ to those at time $n+1$. The method is
introduced for Bayesian inverse problems
in \cite{kantas2014sequential} where it
is used to study the  
problem of inferring the initial condition of the 
Navier-Stokes equations from data. In \cite{beskos2014sequential} 
the method is applied to the inverse problem of determining the
coefficient of a divergence form elliptic PDE from linear
functionals of the solution; furthermore the
method is also proved to converge in the large particle limit
$J \to \infty.$

In practice the SMC method can perform poorly. This happens 
when the weights $\{w_n^{(j)}\}_{j=1}^J$ degenerate in that
 one of the weights 
takes a value close to one and all others are negligible. The
EnKF aims to counteract this by always seeking an approximation
in the form
\begin{equation}
  {\mu_{n}}\simeq \frac 1J \sum_{j=1}^J \delta_{u_n}^{(j)}
  \end{equation}
and thus
\begin{equation}
  {\mu_{n+1}}\simeq \frac 1J \sum_{j=1}^J \delta_{u_{n+1}}^{(j)}\;.
  \end{equation}
The method is defined by the mapping of the particles at time $n$
into those at time $n+1$. Let $u_n=\{u_n^{(j)}\}_{j=1}^J$.
Then using equation (25) in \cite{iglesias2013ensemble} with 
$\Gamma \mapsto h^{-1}\Gamma$ 
shows that this mapping of particles has the form 
 \beq \label{eq:EnKFiter}
 u_{n+1}^{(j)}=u_{n}^{(j)}+C^{up}(u_n)(C^{pp}(u_n)+h^{-1}\Gamma)^{-1}
\bigl(y_{n+1}^{(j)}-\mathcal G(u_n^{(j)})\bigr), \quad j=1,\cdots, J,
 \eeq
 where
$$y_{n+1}^{(j)}=y+\xi^{(j)}_{n+1}$$
and, for $u=\{u^{(j)}\}_{j=1}^J$, we define the operators 
$C^{pp}$ and $C^{up}$ by
 \begin{align}
  \qquad C^{pp}(u)&=\frac{1}{J}\sum_{j=1}^J\bigl(\cG(u^{(j)})-\bG\bigr)
\otimes\bigl(\cG(u^{(j)})-\bG\bigr), \label{eq:sa}\\
  \qquad C^{up}(u)&=\frac{1}{J}\sum_{j=1}^J\bigl(u^{(j)}-\bu\bigr)
\otimes\bigl(\cG(u^{(j)})-\bG\bigr), \label{eq:sb}\\
  \qquad \bu&=\frac1J\sum_{j=1}^Ju^{(j)}, \quad \bG=\frac1J\sum_{j=1}^J
\cG(u^{(j)}). \label{eq:sc} 
  \end{align}
We will consider both the cases where $\xi^{(j)}_{n+1} \equiv 0$
and where, with respect to both $j$ and $n$, the
$\xi^{(j)}_{n+1}$ are i.i.d. random variables 
distributed according to $N(0,h^{-1}\Gamma)$.
We can unify by considering the i.i.d. family of
random variables $\xi^{(j)}_{n+1} \sim N(0,h^{-1}\Sigma)$
and focussing exclusively on the cases where $\Sigma=\Gamma$
and where $\Sigma=0.$ \cls{The theoretical results will be solely derived for the setting $\Sigma=0$, i.e. no artificial noise will be added to the observational data.}

The derivation of the EnKF as presented here relies on a Gaussian approximation\cls{,
which can be interpreted as a linearization of the nonlinear operator $L_n$ in the following way: in the large ensemble size limit, the EnKF estimate corresponds to the best linear estimator of the conditional mean. See \cite{goldstein} for a general
discussion of Bayes linear methods and \cite{law2016deterministic} for details in
the context of data assimilation.} 
Thus, besides the approximation of the measures $\mu_n$ by a $J$ 
particle Dirac measure, there is an additional \cls{uncontrolled error 
resulting from the Gaussian approximation, and analyzed 
in \cite{ernst2015analysis}.
In \cite{stordal}, the EnKF in combination with an annealing process is 
used to account for nonlinearities
 in the forward problem by weight-correcting the EnKF. For data assimilation problems, 
similar techniques can be applied to improve the performance of the EnKF in the nonlinear regime, see e.g. \cite{bocquet2014iterative} and the references therein for more details.}

\cls{In summary, except in the Gaussian} case of linear problems, there is no convergence to
$\mu_n$ as $J \to \infty$. Our focus, then, is on understanding the
properties of the algorithm for fixed $J$, as an optimization method;
we do not study the approximation of the measure $\mu.$ In this context
we also recall here the invariant subspace property of the
EnKF method, as established in \cite{iglesias2013ensemble}:

\begin{lemma}
\label{lem:sub}
If ${\cal S}$ is the linear span of $\{u_0^{(j)}\}_{j=1}^J$
then $u_n^{(j)} \in {\cal S}$ for all $(n,j) \in {\mathbb Z}^+ \times
\{1,\dots, J\}.$
\end{lemma}

\section{Continuous Time Limit}
\label{sec:C}

Here we study a continuous time limit of the EnKF methodology
as applied to inverse problems; 
this limit arises by taking the parameter $h$, appearing in the
incremental formulation \eqref{eq:bipi} of the Bayesian inverse 
problem \eqref{eq:bip}, to zero. We proceed purely formally, with no
proofs of the limiting process, 
as our main aim in this paper is to study the behavior of the
continuous time limits, not to justify taking that limit. However
we note that the invariant subspace property of Lemma \ref{lem:sub}
means that the desired limit theorems are essentially finite
dimensional and standard methods from numerical analysis may
be used to establish the limits. \cls{In the next section all the theoretical results 
are derived under the assumption that $\mathcal G$ is linear and $\Sigma=0$. However
in this section we derive the continuous time limit in a general setting, before
specifying to the linear noise-free case.}

\subsection{The Nonlinear Problem}
\label{ssec:NL}

\cls{Recall the definition of the operator $C^{up}$ given by
\eqref{eq:sb}.}
We recall that $u_n=\{u_n^{(j)}\}_{j=1}^J$,
and assume that $u_n \approx u(nh)$ in \eqref{eq:EnKFiter} in the limit
$h \to 0.$ The update step of the EnKF \eqref{eq:EnKFiter} can be written in 
the form of a time-stepping scheme:
 \begin{eqnarray*}
 u_{n+1}^{(j)}&=&u_{n}^{(j)}+hC^{up}(u_{n})(hC^{pp}(u_{n})+\Gamma)^{-1}
\bigl(y-\mathcal G(u_n^{(j)})\bigr)\\&&+hC^{up}(u_n)(hC^{pp}(u_n)+\Gamma)^{-1}
\xi_{n+1}^{(j)}\\
&=&u_{n}^{(j)}+hC^{up}(u_n)(hC^{pp}(u_n)+\Gamma)^{-1}
\bigl(y-\mathcal G(u_n^{(j)})\bigr)\\&&+h^{\frac12}C^{up}(u_n)(hC^{pp}(u_n)+\Gamma)^{-1}
\sqrt{\Sigma}\zeta_{n+1}^{(j)}\;,
\end{eqnarray*}
where $\zeta_{n+1}^{(j)}\sim\cls{\mathcal N}(0,I)$ i.i.d..
If we take the limit $h \to 0$ then this is clearly a \cls{tamed} 
Euler-Maruyama type discretization of the set of coupled It\^{o} SDEs
   \begin{eqnarray}\label{eq:ode}
   \frac{\dd u^{(j)}}{\dd t}&=&C^{up}(u)\Gamma^{-1}
(y-\mathcal G(u^{(j)}))+C^{up}\cls{(u)\Gamma^{-1}\sqrt{\Sigma} }
\frac{dW^{(j)}}{dt}.
\end{eqnarray}
Using the definition of the operator $C^{up}$ we see that 
\begin{eqnarray}\label{eq:ode2}
   \frac{\dd u^{(j)}}{\dd t}
   &=&\frac{1}{J}\sum_{k=1}^J \bigl\langle \cG(u^{(k)})-\bG,
y-\cG(u^{(j)})+\sqrt{\Sigma} \frac{dW^{(j)}}{dt}\bigr\rangle_{\Gamma}
\bigl(u^{(k)}-\bu\bigr),
  \end{eqnarray}
where
$$\bu=\frac1J \sum_{j=1}^J u^{(j)}, \quad \bG=\frac1J \sum_{j=1}^J\mathcal G(u^{(j)})$$
and
$\langle \cdot, \cdot \rangle_{\Gamma}=\langle \Gamma^{-\frac12}\cdot, \Gamma^{-\frac12}\cdot \rangle$ with $\langle \cdot, \cdot \rangle$ the
inner-product on $\cY$. 
The $W^{(j)}$ are independent cylindrical Brownian motions on $\cX$.
The construction demonstrates that, provided a solution to 
\eqref{eq:ode2} exists,
it will satisfy a generalization of the subspace property 
of Lemma \ref{lem:sub} to continuous time \cls{because the vector field
is in the linear span of the ensemble itself.}

\subsection{The Linear Noise-Free Problem}
\label{ssec:L}

In this subsection we study the linear inverse problem, 
for which $\cG(\cdot)=A\cdot$ for some $A \in {\mathcal L}(\cX,\cY)$.
We also restrict attention to the case where $\Sigma=0$.
\cls{Then the continuous time limit equation \eqref{eq:ode2} }
becomes
   \begin{equation}\label{eq:lode}
   \frac{\dd u^{(j)}}{\dd t}=\frac{1}{J}\sum_{k=1}^J \bigl\langle A(u^{(k)}-\bu),
y-Au^{(j)}\bigr\rangle_{\Gamma} \bigl(u^{(k)}-\bu\bigr), \quad j=1,\cdots, J.
  \end{equation}
For $u=\{u^{(j)}\}_{j=1}^J$ we define the empirical covariance operator 
$$\qquad C(u)=\frac{1}{J}\sum_{k=1}^J\bigl(u^{(k)}-\bu\bigr)\otimes
\bigl(u^{(k)}-\bu\bigr).$$ 
Then equation \eqref{eq:lode} may be written in the form
\begin{equation}\label{eq:lode2}   \frac{\dd u^{(j)}}{\dd t}=-C(u)D_{u}\Phi(u^{(j)};y)
\end{equation}
where, in this linear case,
\begin{equation}
\label{eq:fil}
\Phi(u;y)=\frac12\|\Gamma^{-\frac12}(y-Au)\|^2.
\end{equation}
Thus each particle performs a preconditioned 
gradient descent for $\Phi(\cdot;y)$, 
and all the individual gradient descents are coupled through the 
preconditioning of the flow by the empirical covariance $C(u).$
This gradient flow is thus nonlinear, even though the forward map 
is linear.
Using the fact that $C$ is positive semi-definite it follows that
\begin{equation}
\label{eq:apb}
\frac{\dd}{\dd t}\Phi\bigl(u(t);y\bigr) \cls{=\frac{\dd}{\dd t} \frac12\|\Gamma^{-\frac12}(y-Au)\|^2} \le 0. 
\end{equation}
This gives an a priori bound on {$\|Au(t)\|_\Gamma$},
but does not give global existence of a solution when $ \Gamma^{-\frac12} A$
is compact. However, global existence may be proved as we show in
in the next section, \cls{using the invariant subspace property.}

\section{Asymptotic Behavior in the Linear Setting}
\label{sec:asy}

In this section we study the differential equations \eqref{eq:lode}.
Although derivation of the continuous time limit suggests \cls{stopping
the integration at time $T=1$} it is nonetheless
of interest to study the dynamical system in the long-time 
asymptotic $T \to \infty$ as this sheds light on the mechanisms 
at play within the ensemble methodology, and points to possible
improvements of the algorithm.

In the first subsection \ref{ssec:nf} we study the case where the
data is noise-free.  Theorem \ref{t:1} shows existence of a solution satisfying
the subspace property; Theorem \ref{t:2} shows collapse of all ensemble
members towards their mean value at an algebraic rate; Theorem
\ref{t:3} decomposes the error, in the image space under $A$, into an error
which decays to zero algebraically (in a subspace determined by the
initial ensemble) and an error which is constant in time (in a 
complementary space); Corollary \ref{cor:rescgc} transfers these results 
to the state space of the unknown under additional assumptions. 
\cls{We assume throughout this section that the forward operator is a bounded, linear operator. The convergence analysis presented in Theorem
\ref{t:3} additionally requires the operator to be injective. For compact operators, the convergence result in the observational space does not imply convergence in the state space. However, assuming the forward operator is boundedly invertible, the generalization is straightforward. Note that this assumption is typically not fulfilled in the inverse setting, but opens up the perspective to use the EnKF as a linear solver.}
Subsection \ref{ssec:noisy} briefly discusses the noisy data case 
where the results are analogous to those in the preceding subsection.

\subsection{The Noise-Free Case}
\label{ssec:nf}

In proving the following theorems we will consider the situation
where the data $y$ is the image of a truth $\ud \in \cX$ under $A$.
It is then useful to define
\begin{align}
&e^{(j)}=u^{(j)}-\bar u, \quad r^{(j)}=u^{(j)}-u^\dagger\\
&E_{lj}=\langle Ae^{(l)},Ae^{(j)}\rangle_{\Gamma}, \quad
 R_{lj}=\langle Ar^{(l)},Ar^{(j)}\rangle_{\Gamma}, \quad
 F_{lj}=\langle Ar^{(l)},Ae^{(j)}\rangle_{\Gamma}.
\end{align} 
We view the last three items as entries of matrices $E,R$ and $F$.
The resulting matrices $E,R,F\in\mathbb R^{J \times J}$ satisfy the
following: (i) $E,R$ are symmetric; (ii) we may factorize
$E(0)=X\Lambda(0)X^T$ where $X$ is an orthogonal matrix whose columns
are the eigenvectors of $E(0)$, and $\Lambda(0)$ is a diagonal
matrix of corresponding eigenvalues; (iii) if $\one=(1,\dots, 1)^T$, 
then $E\one=F\one=0.$ 
Note that $e^{(j)}$ measures deviation of the $j^{th}$ ensemble member from
the mean of the entire ensemble, and $r^{(j)}$ measures deviation of 
the $j^{th}$ ensemble member from the truth $\ud$ underlying the data. 
The matrices $E,R$ and $F$ contain information about
these deviation quantities, when mapped forward under the operator $A$.
The following theorem establishes existence and uniqueness 
of solutions to \eqref{eq:lode}.

\begin{theorem}
\label{t:1}
Assume that $y$ is the image of a truth $\ud \in \cX$ under $A$. Let $u^{(j)}(0) \in \cX$
for $j=1,\dots, J$ and define $\cX_0$
to be the linear span of the $\{u^{(j)}(0)\}_{j=1}^J.$ Then equation \eqref{eq:lode} has a unique solution 
$u^{(j)}(\cdot) \in C([0,\infty);\cX_0)$ for $j=1,\dots, J.$
\end{theorem}

\begin{proof}
It follows from \eqref{eq:lode} that
\begin{equation}
\label{eq:u}
\frac{\dd  u^{(j)}}{\dd t}= -\frac{1}{J}\sum_{k=1}^J F_{jk}e^{(k)}
\end{equation}

The right-hand side of equation \eqref{eq:lode2} is locally
Lipschitz in $u$ as a mapping from $\cX_0$ to $\cX_0$.
Thus local existence of a solution in $C([0,T);\cX_0)$ holds for some $T>0$, 
since $\cX_0$ is finite dimensional. 
Thus it suffices to show that the solution
does not blow-up in finite time. To this end we prove in Lemma
\ref{l:EF},\cls{ which is presented in the Appendix \ref{appendix},} that the matrices $E(t)$ and $F(t)$ are bounded by
a constant depending on initial conditions, but not on time $t$.
Using the global bound on $F$ it follows from \eqref{eq:u} that $u$
is globally bounded by a constant depending on initial conditions,
and growing exponentially with $t$. Global existence for $u$
follows and the theorem is complete. 
\end{proof}

The following theorem shows that all ensemble members collapse towards
their mean value at an algebraic rate; and it demonstrates that
the collapse slows down linearly as ensemble size increases. 

\begin{theorem}
\label{t:2}
Assume that $y$ is the image of a truth $\ud \in \cX$ under $A$. Let $u^{(j)}(0) \in \cX$
for $j=1,\dots, J$. Then the matrix valued quantity $E(t)$ converges to $0$ for $t \to \infty $ and, indeed \cls{$\|E(t)\|={\mathcal O}(Jt^{-1}).$} 
\end{theorem}

\begin{proof}
Lemma \ref{l:EF},\cls{ which is presented in the Appendix \ref{appendix},} shows that the quantity $E(t)$ satisfies the differential equation
\begin{equation}
\frac{\dd  }{\dd t}E=-\frac2JE^2
\end{equation}
with initial condition $E(0)=X\Lambda(0)X^\top$, $\Lambda(0)=\mbox\{\lambda_0^{(1)}, \ldots , \lambda_0^{(J)} \}$ and orthogonal $X$. Using the eigensystem $X$ gives the solution $E(t)=X\Lambda(t)X^\top$, where the entries of the diagonal matrix $\Lambda(t)$ are given by $({\frac 2J t+\frac{1}{\lambda_0^{(j)}}})^{-1}$ if $\lambda_0^{(j)}\neq 0$ and $0$ otherwise. Then, the claim immediately follows from the explicit form of the solution $E(t)$.
\end{proof}
We are now interested in the long-time behavior of the residuals with
respect to the truth. Theorem \ref{t:3} characterizes the relation between the approximation quality of the initial ensemble and the convergence behavior of the residuals.

\begin{theorem}
\label{t:3}
Assume that $y$ is the image of a truth $\ud \in \cX$ under $A$ {and the forward operator $A$ is one-to-one}. 
Let {$\Yl$} denote the linear span of the $\{{A}e^{(j)}(0)\}_{j=1}^J$
and let {$\Yp$} denote the orthogonal complement of {$\Yl$}
in {$\cY$} 
with respect to the inner product $\langle\cdot,\cdot\rangle_\Gamma$
and assume that the initial ensemble members are chosen
so that {$\Yl$} has the maximal dimension {$\min\{J-1,\dim({\mathcal{Y}})\}.$}  
Then ${A}r^{(j)}(t)$ may be decomposed uniquely
as {$A\rl(t)+A\rp(t)$ with $A\rl \in \Yl$ and $A\rp \in \Yp$}.
Furthermore, for all $j \in \{1,\cdots, J\}$, 
${A\rl}(t) \to 0$ as $t \to \infty$ and, for
all $j \in \{1,\cdots, J\}$ and $t \ge 0$,
${A}\rp(t)={A}\rp(0)={Ar^{(1)}_{\perp}}(0).$
\end{theorem}

\begin{proof}
Lemma \ref{l:er},\cls{ which is presented in the Appendix \ref{appendix},} shows that the matrix $L$, the linear transformation
which determines how to write $\{Ae^{(j)}(t)\}_{j=1}^J$ in terms 
of the coordinates $\{Ae^{(j)}(0)\}_{j=1}^J$, is invertible for all $t \ge 0$
and hence that the linear span of the $\{{A}e^{(j)}(t)\}_{j=1}^J$ is equal
to $\Yl$ for all $t \ge 0.$ Lemma \ref{l:er} also shows that
${A}r^{(j)}(t)$ may be decomposed uniquely
as ${A}\rl(t)+{A}\rp(t)$ with ${A}\rl \in \Yl$ and ${A}\rp \in \Yp$ and that
${A}\rp(t)={A}\rp(0)$ for all $t \ge 0.$ It thus remains to show that
${A}\rl(t)$ converges to zero as $t \to \infty.$

From Lemma \ref{l:er} we know that we may write
\begin{equation}
\label{eq:key}
{A}r^{(j)}(t)=\sum_{k=1}^J\alpha_k {A}e^{(k)}(t)+{Ar^{(1)}_{\perp}},
\end{equation}
for some ($j$-dependent)
coefficients $\alpha=(\alpha_1, \dots, \alpha_J)^T \in \mathbb R^J$. 
Furthermore Lemma \ref{l:er} shows that if {we initially choose $\alpha$ to be
orthogonal to the eigenvectors $x^{(k)},\ k=1,\ldots, J-\tilde J$ of $E$ with corresponding eigenvalues $\lambda^{(k)}(t)=0,\ k=1,\ldots, J-\tilde J$,} then this property will be preserved
for all time. {This is since we have $QL^{-1}x^{(k)}=Qx^{(k)}=0,\ k=1,\ldots, J-\tilde J$ and 
the matrix $\Upsilon$ with $j^{th}$ row given by $\alpha$ satisfies 
$Q=\Upsilon L$ so that $\Upsilon x^{(k)} =0,\ k=1,\ldots, J-\tilde J.$ Finally we choose 
coordinates in which $A\rp(0)$ is orthogonal to $\Yl.$}

Define the seminorms
  \begin{eqnarray*}
|\alpha|_1^2&:=&|E^{1/2}\alpha|^2\\
|\alpha|_2^2&:=&|E\alpha|^2\;.
 \end{eqnarray*}
Note that that these norms are time-dependent because $E$ is.
On the subspace of $\bbR^J$ orthogonal to {$\spann\{x^{(1)},\ldots,x^{(J-\tilde J)}\}$},
   \begin{equation}
\label{eq:deen1}
|\alpha|_2^2 \ge \lambda_{\min}(t)|\alpha|_1^2\;,
  \end{equation}
where $\lambda_{\min}(t)=(\frac 2J t+\frac{1}{\lambda_0^{\min}})^{-1}$ is the minimal positive eigenvalue of $E$, see \eqref{eq:odeenslambda3}. 
Furthermore, for the Euclidean norm $|\cdot|$, we have
    \begin{equation}
\label{eq:deen2}
|\alpha|_1^2 \ge \lambda_{\min}(t)|\alpha|^2\;
  \end{equation}
on the subspace of $\bbR^J$ orthogonal to {$\spann\{x^{(1)},\ldots,x^{(J-\tilde J)}\}$}.

Now note that the following differential equation holds 
for the quantity $\|Ar^{(j)}\|_{\Gamma}^2:$
 \begin{equation}
\label{eq:deen}
\frac 12 \frac{\dd  }{\dd t} \|Ar^{(j)}\|_{\Gamma}^2= -\frac1J \sum_{r=1}^J \langle Ar^{(j)},Ae^{(r)}\rangle_{\Gamma}\langle Ar^{(j)},Ae^{(r)}\rangle_{\Gamma}.
\end{equation}
We also have
     \begin{eqnarray*}
     \sum_{r=1}^J\langle Ar^{(j)}, Ae^{(r)}\rangle_{\Gamma}^2&=&\sum_{r=1}^J\langle \sum_{k=1}^J\alpha_kAe^{(k)}, Ae^{(r)}\rangle_{\Gamma}\langle \sum_{l=1}^J\alpha_lAe^{(l)}, Ae^{(r)}\rangle_{\Gamma}\\
     &=&\sum_{k=1}^J\sum_{l=1}^J \alpha_k \left(\sum_{r=1}^JE_{kr}E_{lr}\right)\alpha_l\\
     &=&|\alpha|_2^2\;.
        \end{eqnarray*}
Using \eqref{eq:key}, the norm of the residuals can be expressed in
terms of the coefficient vector of the residuals as follows:  
\begin{eqnarray*}
\|Ar^{(j)}\|_{\Gamma}^2&=&\langle \sum_{k=1}^J\alpha_kAe^{(k)}+A\rl(0), 
\sum_{l=1}^J\alpha_lAe^{(l)}+A\rl(0)\rangle_{\Gamma}\\
                                &=&\sum_{k=1}^J\sum_{l=1}^J \alpha_k E_{kl}\alpha_l+\|A\rl(0)\|_{\Gamma}^2\\
&=&|\alpha|^2_1+\|A\rl(0)\|_{\Gamma}^2\;.
 \end{eqnarray*}

Thus, the coefficient vector satisfies the following differential equation
  \begin{eqnarray*}
\frac 12 \frac{\dd  }{\dd t} |\alpha|_1^2&=&- \frac1J |\alpha|_2^2\le-\frac{\lambda_{\min}(t)}{J}|\alpha|_1^2\;,
 \end{eqnarray*}
 which gives
   \begin{eqnarray*}
\frac{1}{|\alpha|_1^2} \frac{\dd  }{\dd t} |\alpha|_1^2&\le&- \frac2J \left(\frac 2J t+\frac{1}{\lambda_0^{\min}}\right)^{-1}
 \end{eqnarray*}
 Hence, using that $
\ln|\alpha(t)|_1^2-\ln|\alpha(0)|_1^2\le \ln \frac{1}{\lambda_0^{\min}}-\ln \left(\frac 2J t+\frac{1}{\lambda_0^{\min}}\right)
 $, the coefficient vector can be bounded by
     \begin{eqnarray*}
|\alpha(t)|_1^2\le \frac{1}{\lambda_0^{\min}}|\alpha(0)|_1^2 \lambda_{\min}(t)\;.
 \end{eqnarray*}
 In the Euclidean norm, we have
 \begin{eqnarray*}
\lambda_{\min}(t)|\alpha(t)|^2\le \frac{1}{\lambda_0^{\min}}|\alpha(0)|_1^2 \lambda_{\min}(t)\;
 \end{eqnarray*}
 and thus
       \begin{eqnarray*}
|\alpha(t)|^2\le \frac{1}{\lambda_0^{\min}}|\alpha(0)|_1^2 \le \frac{\lambda_0^{(J)}}{\lambda_0^{\min}}\|\alpha(0)\|^2 
\;.
 \end{eqnarray*}
Since $\alpha$ is bounded in time, and since the $e^{(k)} \to 0$ by
Lemma \ref{l:er}, the desired result follows from \eqref{eq:key}.
\end{proof}

Corollary \ref{cor:rescgc} generalizes the convergence results of the preceding theorem to the infinite dimensional setting, i.e. to the case $\dim \cY=\infty$ \cls{under the additional assumption that the forward operator $A$ is boundedly invertible. This implies that $A$ cannot be a compact operator. However, this result allows to use the EnKF as a linear solver, since the convergence results can be transferred to the state space under the stricter assumptions on $A$.} 
\begin{corollary}\label{cor:rescgc}
Assume that $y$ is the image of a truth $\ud \in \cX$ under $A$, the forward operator $A$ is boundedly invertible and the initial ensemble is chosen such that the subspace $\cX_0=\spann\{e^{(1)},\ldots,e^{(J)}\}$ has maximal dimension $J-1$. Then, there exists a unique decomposition of the residual $r^{(j)}(t)=\rl(t)+\rp(t)$ with $\rl \in \cX_0$ and $\rp \in \cX_1$, where $\cX=\cX_0\bigoplus \cX_1$. The two subspaces $\cX_0$ and $\cX_1$ are orthogonal 
with respect to the inner product $\langle\cdot,\cdot\rangle_{\Xp}:=\langle \Gamma^{-\frac12} A\cdot,\Gamma^{-\frac12} A\cdot \rangle $. Then,  ${\rl}(t) \to 0$ as $t \to \infty$ and $\rp(t)=\rp(0)=r^{(1)}_{\perp}(0).$
\end{corollary}
\begin{proof}
The assumption that the forward operator is boundedly invertible ensures that the range of the operator is closed. \cls{Furthermore, the invertibility of the operator $A$ allows to transfer results from the observational space directly to the state space.} Thus, the same arguments as in the proof of Theorem \ref{t:3} prove the claim.
\end{proof}

\subsection{Noisy Observational Data}
\label{ssec:noisy}

Very similar analyses to those in the previous subsection
may be carried out in the case where the observational data $y^\dagger$ 
is polluted by additive noise $\eta^\dagger \in \mathbb R^K:$ 
\[
y^\dagger=Au^\dagger+\eta^\dagger\;.
\] 
Global existence of solutions, and ensemble collapse,
follow similarly to the proofs of Theorems \ref{t:1}
and \ref{t:2}.
Theorem \ref{t:3} is more complex to generalize since
it is not the mapped residual $Ar^{(j)}$ which is decomposed into
a space where it decays to zero and a space where it
remains constant, but rather the quantity $\vartheta^{(j)}=Ar^{(j)}-\eta^{\dagger}.$
Driving this quantity to zero of course leads to over-fitting. 
Furthermore, the generalization to the infinite dimensional setting 
as presented in Corollary \ref{cor:rescgc} is no longer valid, 
since the noise \cls{may} take the data out of the 
range of the forward operator.

\section{Numerical Results}
\label{sec:N}

In this section we present numerical experiments both with and without
data, illustrating the theory of the previous section for the linear
inverse problem. We also study a nonlinear groundwater flow inverse 
problem, demonstrating that the theory in the linear problem provides 
useful insight for the nonlinear problem.

\subsection{Linear Forward Model}
We consider the one dimensional elliptic equation
 \begin{equation}\label{eq:forward}
    -\frac{\rd^2 p}{\rd x^2}+p=u \quad \mbox{in } D:=(0,\pi)\, , \ p=0  \quad \mbox{in } \partial D\; ,
\end{equation}
where the uncertainty-to-observation operator is given by $\mathcal G=\mathcal O \circ G=\mathcal O \circ A^{-1}$ with $A=-\frac{\rd^2 }{\rd x^2}+id$ and $D(A)=H^2(I) \cap H^1_0$. \cls{The observation operator $\mathcal O$ consists of $K=2^{4}-1$ system responses at $K$ equispaced observation points at $x_k=\frac{k}{2^{4}}, k=1,\ldots,2^{4}-1$, $o_k(\cdot):=\delta(\cdot -x_k)$, i.e. $\mathcal (O(\cdot))_k=o_k(\cdot)$. }
The forward problem \eqref{eq:forward} is \cls{solved numerically} by a FEM using continuous, piecewise linear ansatz functions on a uniform mesh with meshwidth $h=2^{-8}$ (the spatial discretization leads to a discretization of $u$, i.e. $u \in \mathbb R^{2^8-1}$).

The goal of \cls{the} computation is to recover the unknown data $u$ from noisy observations
\begin{eqnarray}\label{eq:modelne}
y^\dagger&=&p+\eta =\mathcal O A^{-1}u^\dagger+\eta\;.
\end{eqnarray}
The measurement noise is chosen to be normally distributed, $\eta \sim \mathcal N(0,\gamma I )$, $\gamma \in \mathbb R\cls{,\ \gamma>0}, \ I \in \mathbb R^{K \times K}$.
The initial ensemble of particles is chosen to be based on the
eigenvalues and eigenfunctions $\{\lambda_j,z_j\}_{j \in \mathbb N}$ 
of the covariance operator $C_0$. Here $C_0=\beta(A-id)^{-1}$, $\beta=10$. 
Although we do not pursue the Bayesian interpretation of the EnKF,
the reader interested in the Bayesian perspective may think of the
prior as being $\mu_0=N(0,C_0)$, \cls{which is a Brownian bridge}. 
In all our experiments we set $u^{(j)}(0)=\sqrt{\lambda_j}\zeta_j z_j$ with $\zeta_j \sim \mathcal N(0,1)$ for $j=1,\ldots,J$. Thus the $j^{th}$
element of the initial ensemble may be viewed as the $j^{th}$ term in
a \cls{Karhunen-Lo\`eve (KL) expansion of a draw from $\mu_0$ which, in this
case, is a Fourier sine series expansion.}

\subsubsection{Noise-Free Observational Data}
To numerically verify the theoretical results presented in section \ref{ssec:nf}, we first restrict the discussion to the noise-free case, i.e. we assume that $\eta=0$ in \eqref{eq:modelne} and set $\Gamma=id$.
The study summarized in Figures \ref{fig:linNFe} - \ref{fig:linNFsol} shows the influence of the number of particles on the dynamical behavior of the quantities $e$ and $r$, the matrix-valued quantities and the resulting EnKF estimate. 
%~\\[-0.1cm]
\begin{center}
\begin{minipage}{0.6\textwidth}
 \begin{figure}[H]
\centering
    \includegraphics[width=1.0\textwidth]{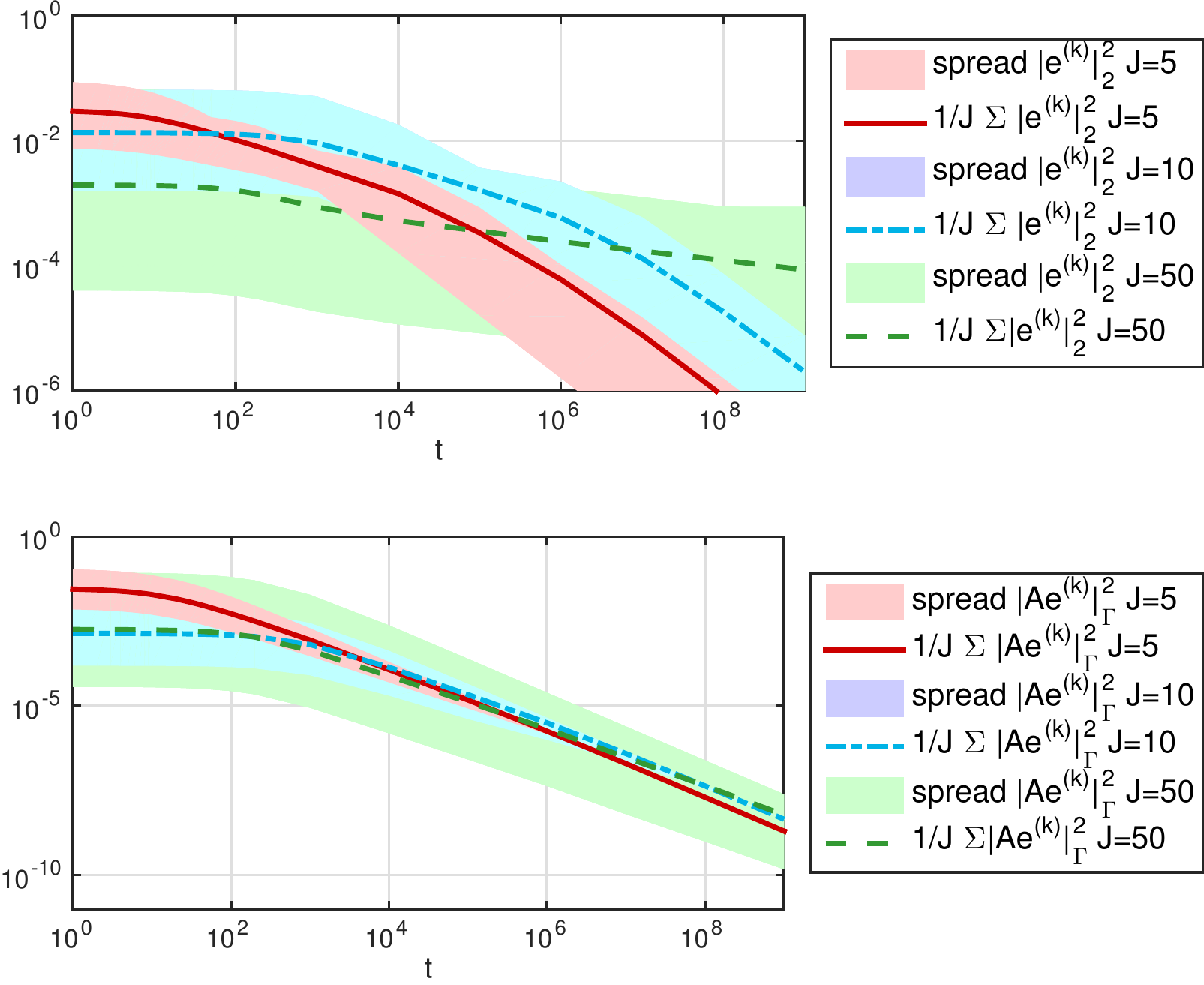}
~\\[-0.75cm]\caption{\footnotesize \label{fig:linNFe}
Quantities $|e|_2^2$, $|Ae|_{\Gamma}^2$ w.r. to time $t$, $J=5$ (red), $J=10$ (blue) and $J=50$ (green), $\beta=10$, $K=2^4-1$, initial ensemble chosen based on KL expansion of $\cls{C_0}=\beta(A-id)^{-1}$. }
 \end{figure}
\end{minipage}
\end{center}
~\\[0.2cm]
As shown in Theorem \ref{t:2}, the rate of convergence of the ensemble collapse is algebraic (cf. Figure \ref{fig:linNFe}) with a constant growing with larger ensemble size. Comparing the dynamical behavior of the residuals, we observe that, for the ensemble of size \cls{J=5}, the estimate can be improved in the beginning, but reaches a plateau after a short time. \cls{Increasing the number of particles to $J=10$ improves the accuracy of the estimate. For the ensemble size $J=50$,} Figure \ref{fig:linNFr} shows the convergence of the projected residuals, i.e. the observations can be perfectly recovered. The same behavior can be observed by comparing the EnKF estimate with the truth and the observational data (cf. Figure \ref{fig:linNFsol}).
% \begin{minipage}{0.1\textwidth}
% ~
% \end{minipage}
~\\[-0.75cm]
\begin{center}
 \begin{minipage}{0.6\textwidth}
 \begin{figure}[H]
\centering
    \includegraphics[width=1.0\textwidth]{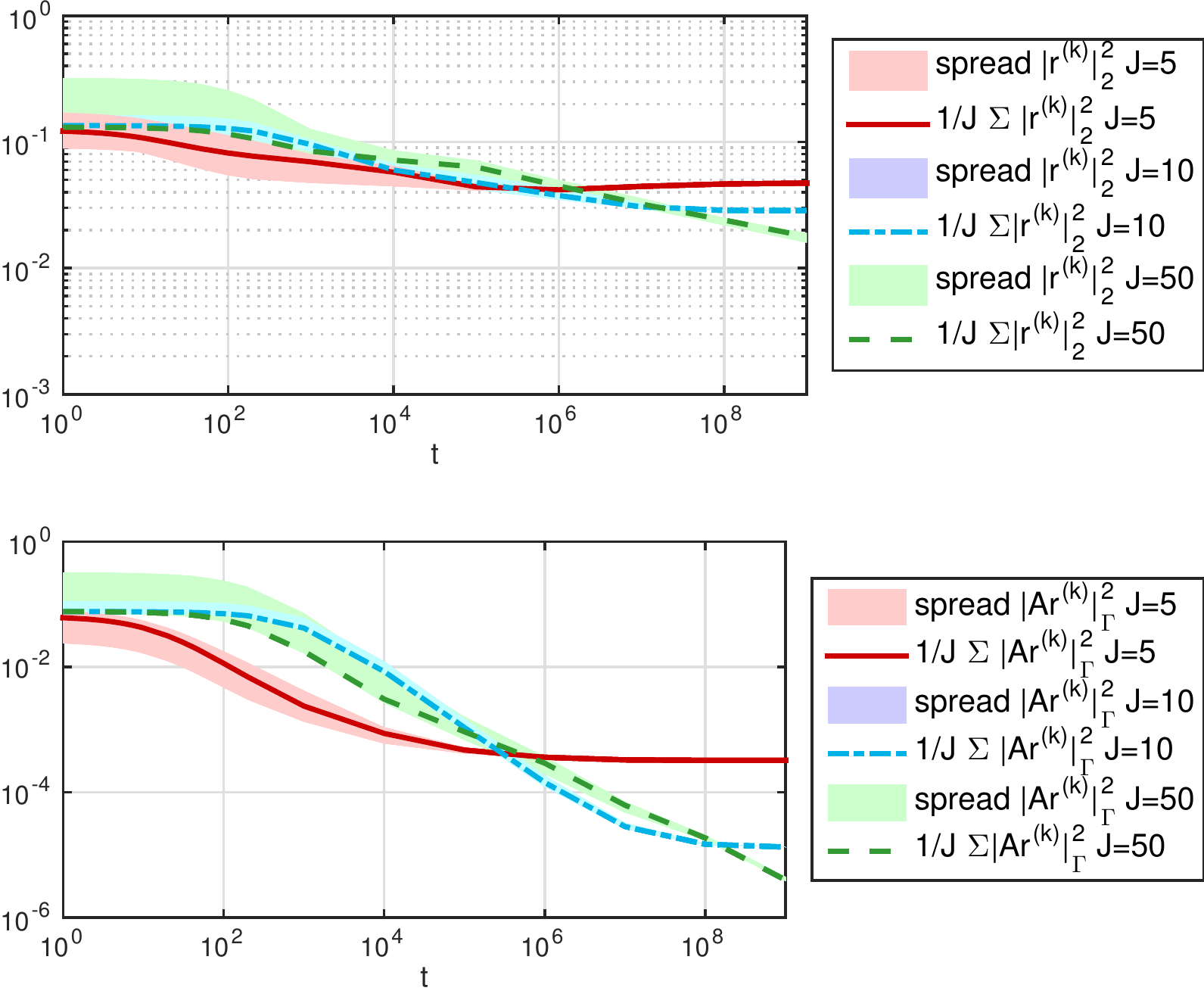}
~\\[-0.75cm]\caption{\footnotesize\label{fig:linNFr}
Quantities $|r|_2^2$, $|Ar|_{\Gamma}^2$ w.r. to time $t$, $J=5$ (red), $J=10$ (blue) and $J=50$ (green), $\beta=10$, $K=2^4-1$, initial ensemble chosen based on KL expansion of $\cls{C_0}=\beta(A-id)^{-1}$. }
 \end{figure}
\end{minipage}
\end{center}
~\\[0.2cm]
The results derived in this paper hold true for each particle, however, for the sake of presentation, the empirical mean of the quantities of interest is shown and the spread indicates the minimum and maximum deviations of the ensemble members from the empirical mean.
~\\[-0.1cm]

 \begin{minipage}{0.45\textwidth}
 \begin{figure}[H]
\centering
~\\[0.15cm]
    \includegraphics[width=\textwidth]{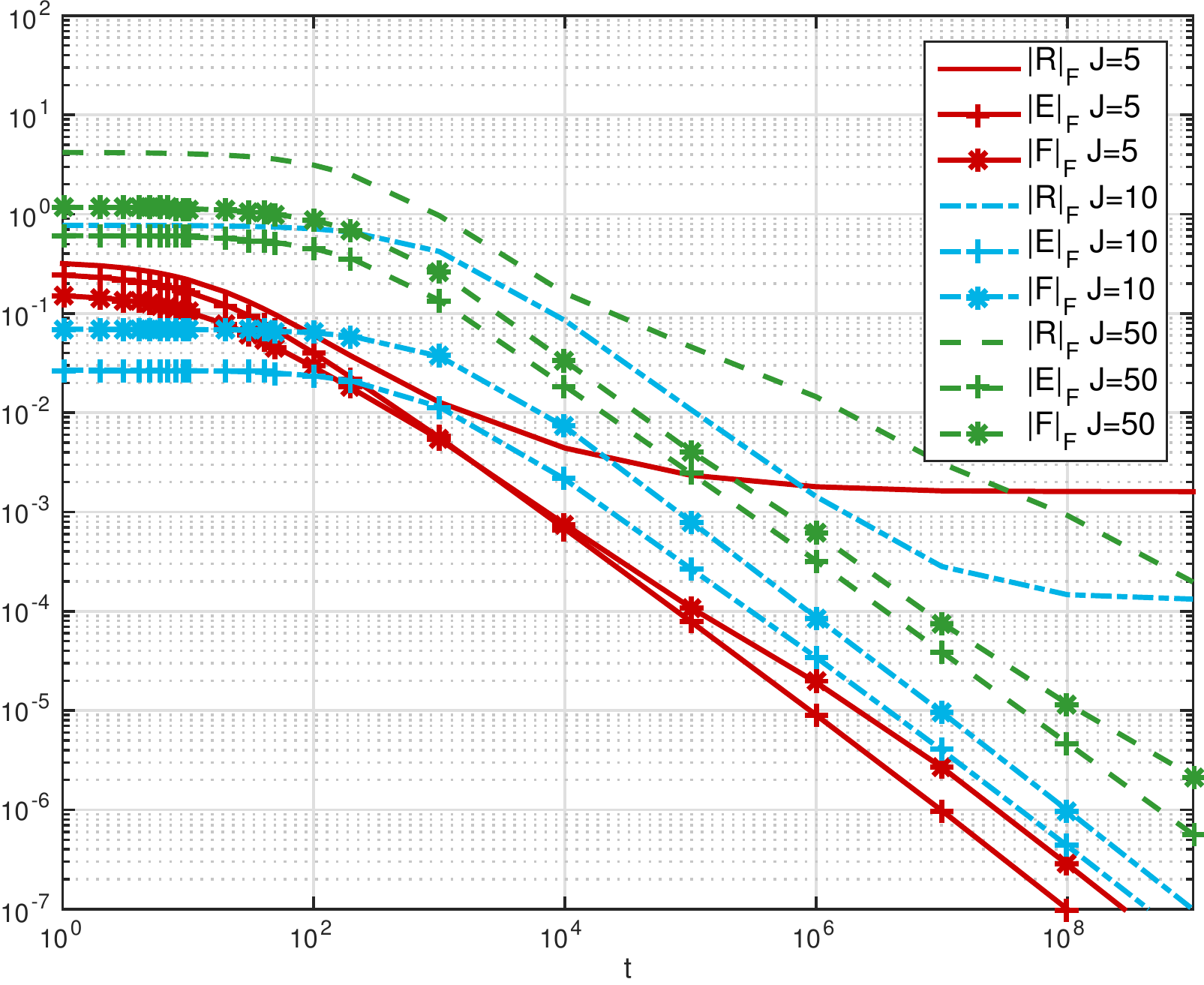}
~\\[-0.75cm]\caption{\footnotesize\label{fig:linNFM}
Quantities $\|E\|_F$, $\|F\|_F$, $\|R\|_F$ w.r. to time $t$, $J=5$ (red), $J=10$ (blue) and $J=50$ (green), $\beta=10$, $K=2^4-1$, initial ensemble chosen based on KL expansion of $\cls{C_0}=\beta(A-id)^{-1}$. \color{white}{Comparison of the EnKF }\color{black}}
 \end{figure}
 \end{minipage}
 \begin{minipage}{0.025\textwidth}
 ~
 \end{minipage}
 \begin{minipage}{0.45\textwidth}
 \begin{figure}[H]
\centering
    \includegraphics[width=\textwidth]{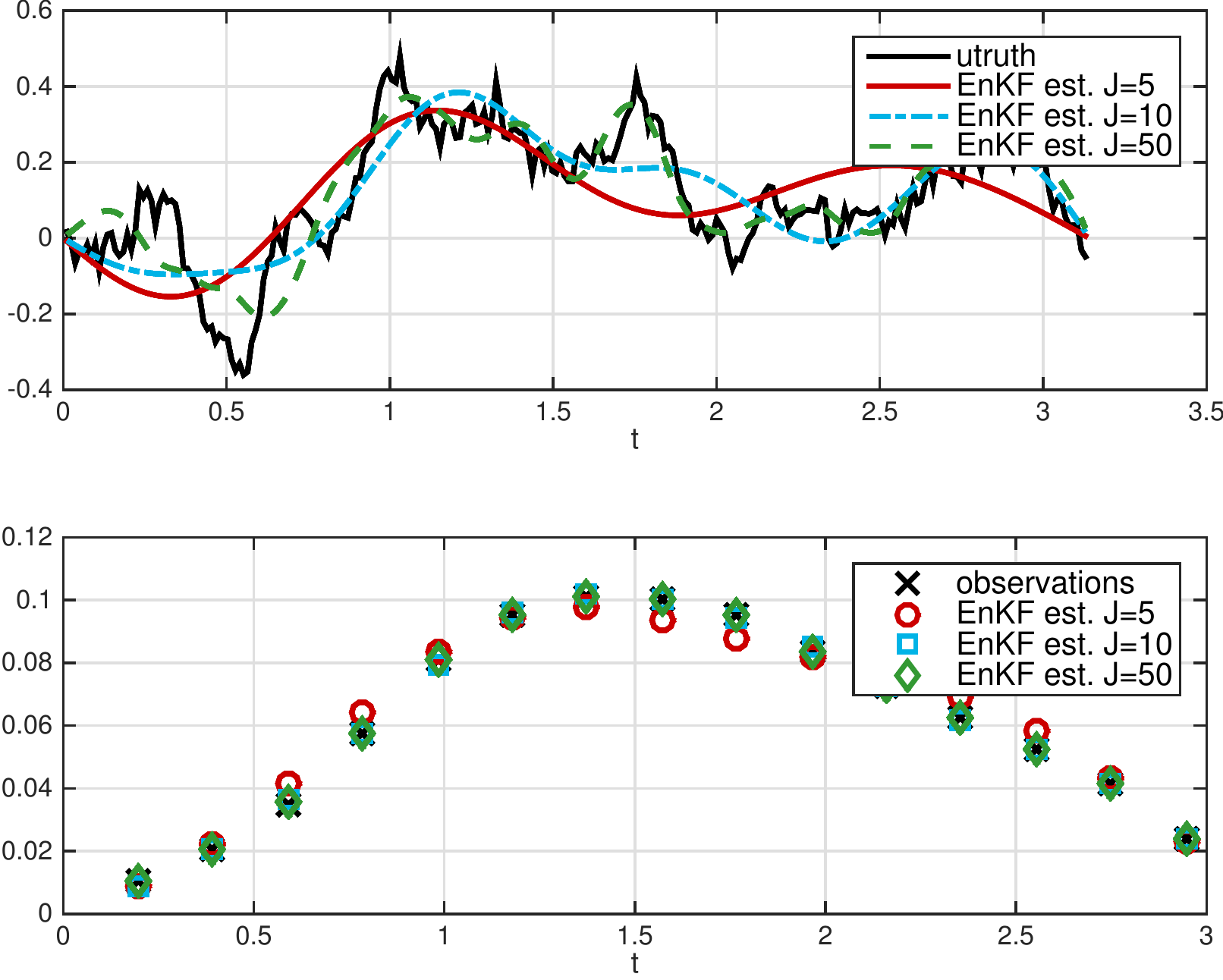}
~\\[-0.75cm]\caption{\footnotesize \label{fig:linNFsol}
Comparison of the EnKF estimate with the truth and the observations, $J=5$ (red), $J=10$ (blue) and $J=50$ (green), $\beta=10$, $K=2^4-1$, initial ensemble chosen based on KL expansion of $\cls{C_0}=\beta(A-id)^{-1}$. }
 \end{figure}

\end{minipage}
~\\[0.2cm]
Due to the construction of the ensembles in the example, the subspace spanned by the ensemble of size 5 is a strict subset of the subspace spanned by the larger \cls{ensembles}. Thus, due to 
Theorem \ref{t:3}, which characterizes the convergence of the residuals with respect to the approximation quality of the subspace spanned by the initial ensemble, the EnKF estimate can be substantially improved by controlling this subspace. As illustrated in Figure \ref{fig:linNFr}, the mapped residual of the ensemble of size 5 decreases monotonically, but levels off after a short time. \cls{Similar convergence properties can be observed for $J=10$.} The same behavior is expected for the larger ensemble, when integrating over a larger time horizon. This can be also observed for the matrix-valued quantities depicted in Figure \ref{fig:linNFM}.  

We will investigate this point further by comparing the performance of two ensembles, both of size 5: one based on the KL expansion and one \cls{chosen such that the contribution of  $A\rp(t)$ in Theorem \ref{t:3} is minimized. Since we use artificial data, we can minimize the contribution of $A\rp(t)$ by ensuring that $Ar^{(1)}=\sum_{k=1}^J \alpha_k Ae^{(k)}$ for some coefficients $\alpha_k\in\mathbb R$. Given $u^{(2)},\ldots,u^{(J)}$ and coefficients $\alpha_1,\ldots,\alpha_J$, we define $u^{(1)}=(1-\alpha_1+\sum_{k=1}^J \alpha_k/J)^{-1}(u^\dagger-\alpha_1/J\sum_{j=2}^J\ u^{(j)}+\sum_{k=2}^J\alpha_k u^{(k)}-\alpha_k/J \sum_{j=2}^J u^{(j)})$, which gives the desired property of the ensemble. Note that this approach is not feasible in practice and has to be replaced by an adaptive strategy \cls{minimizing the contribution of} $A\rp(t)$. However this experiment serves to illusrate the important role of the initial ensemble in determining the error and is included for this reason.}
The convergence rate of the mapped residuals and of the ensemble collapse is algebraic in both cases, with rate 1 (in the squared Euclidean norm). Figure \ref{fig:linNFradaptive} shows the convergence of the projected residuals for the adaptively chosen ensemble. The decomposition of the residuals (cf. Figure \ref{fig:linNFdecompadaptive}) numerically verifies the presented theory, which motivates the adaptive construction of the ensemble. Methods to realize this strategy, in the linear and nonlinear case, will be addressed in a subsequent paper. 
~\\[-0.1cm]

 \begin{minipage}{0.45\textwidth}
 \begin{figure}[H]
\centering
    \includegraphics[width=\textwidth]{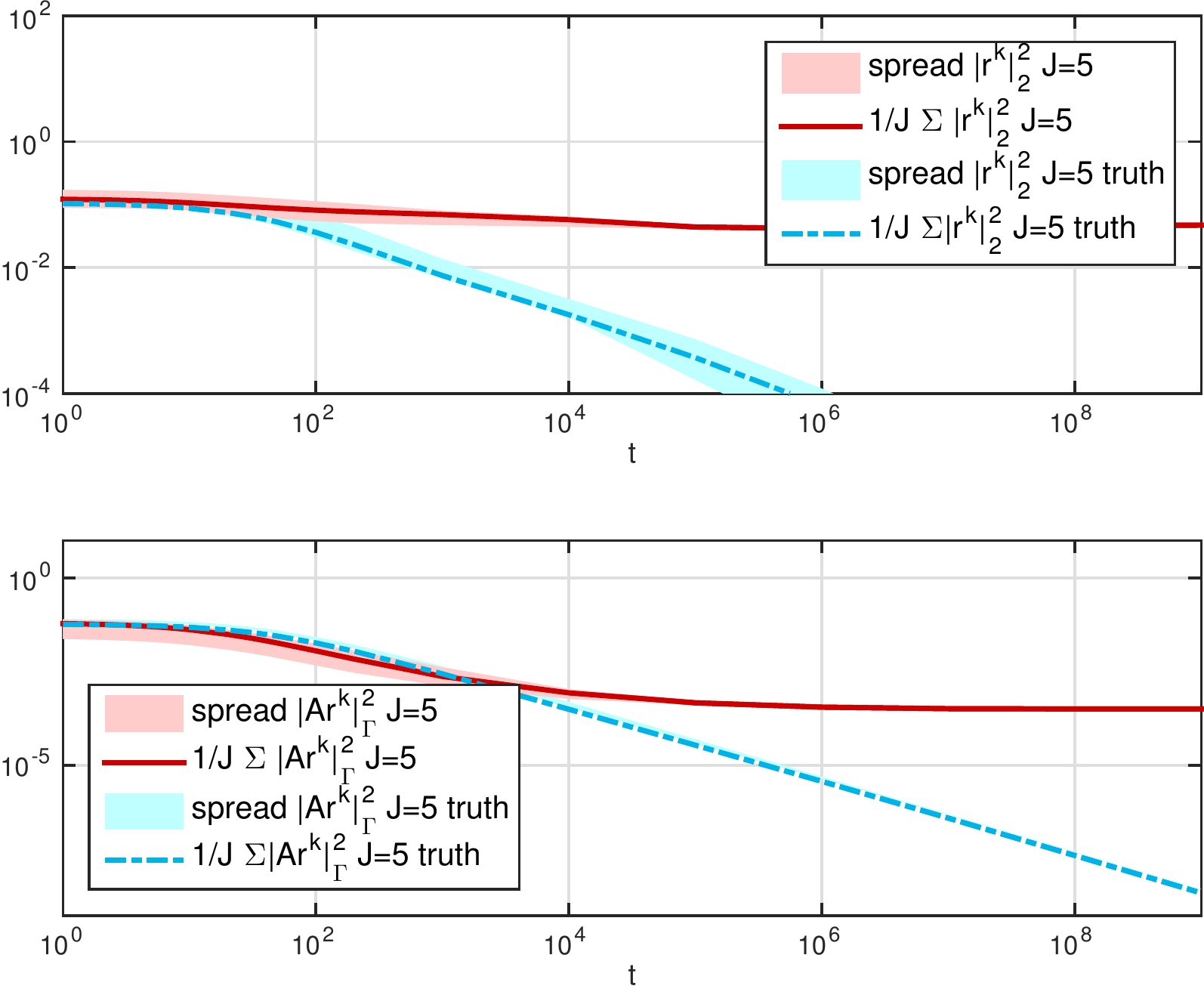}
~\\[-0.75cm]\caption{\footnotesize\label{fig:linNFradaptive}
Quantities $|r|_2^2$, $|Ar|_{\Gamma}^2$ w.r. to time $t$, $J=5$ based on KL expansion  of $\cls{C_0}=\beta(A-id)^{-1}$ (red) and $J=5$ \cls{minimizing the contribution of} $A\rp(t)$ (blue), $\beta=10$, $K=2^4-1$.  \color{white}{Comparison of the EnKF estimate }\color{black}}
 \end{figure}
\end{minipage}
 \begin{minipage}{0.025\textwidth}
 ~
 \end{minipage}
 \begin{minipage}{0.45\textwidth}
 \begin{figure}[H]
\centering
    \includegraphics[width=\textwidth]{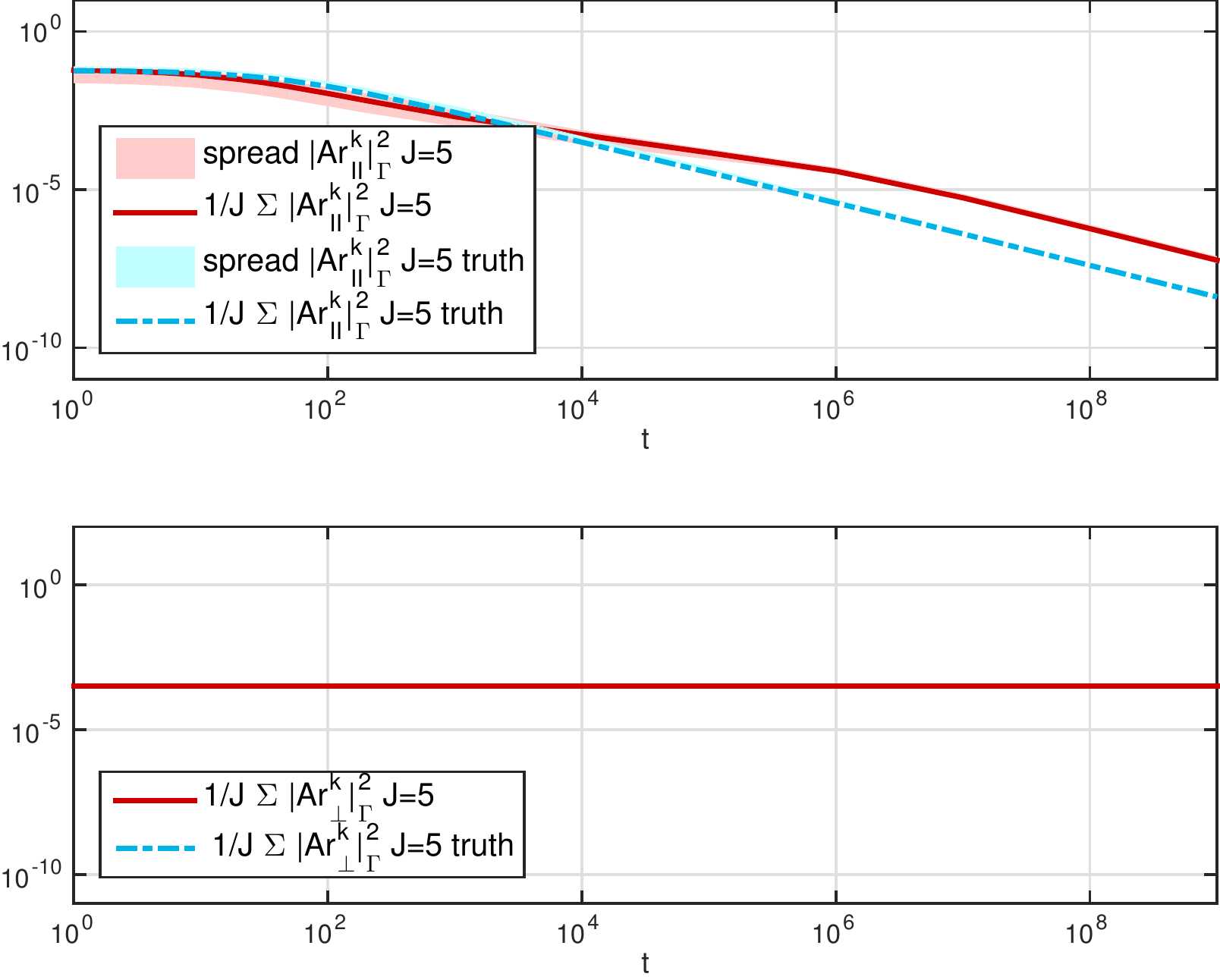}
~\\[-0.75cm]\caption{\footnotesize\label{fig:linNFdecompadaptive}
Quantities $|A\rl|_\Gamma^2$, $|A\rp|_{\Gamma}^2$ w.r. to time $t$, $J=5$ based on KL expansion  of $\cls{C_0}=\beta(A-id)^{-1}$ (red) and $J=5$ \cls{minimizing the contribution of} $A\rp(t)$ (blue), $\beta=10$, $K=2^4-1$.}
 \end{figure}
 \end{minipage}
~\\[0.2cm]

\subsubsection{Noisy Observational Data}\label{sec:NEnoisy}
We will now allow for noise in the observational data, i.e. we assume that the data is given by $y^\dagger=\mathcal O A^{-1} u^\dagger +\eta^\dagger$, where $\eta^\dagger$ is a fixed realization of the random vector $\eta\sim\cls{\mathcal N}(0,0.01^2 \id)$. \cls{Note that the standard deviation is chosen to be roughly 10\% of the (maximum of the) observed data.} \cls{Besides} the quantities $e$ and $r$, the misfit 
$\vartheta^{(j)}=Au^{(j)}-y^\dagger=Ar^{(j)}-\eta^\dagger$ 
of each ensemble member is of interest, since, in practice, the residual is not accessible and the misfit is used to check for convergence and to design an appropriate stopping criterion.

\cls{Besides} the two ensembles with 5 particles introduced in the previous section, we define an additional one \cls{minimizing the contribution of} $\dpp$ to the misfit,
analogously to what was done in the adaptive initialization in the
previous subsection, and motivated by the analogue of Theorem \ref{t:3}
in the noisy data case. Note that the design of an adaptive ensemble based on the decomposition of the projected residual is, in general, not feasible without explicit knowledge of the noise $\eta^\dagger$.

 \begin{minipage}{0.45\textwidth}
 \begin{figure}[H]
\centering
    \includegraphics[width=\textwidth]{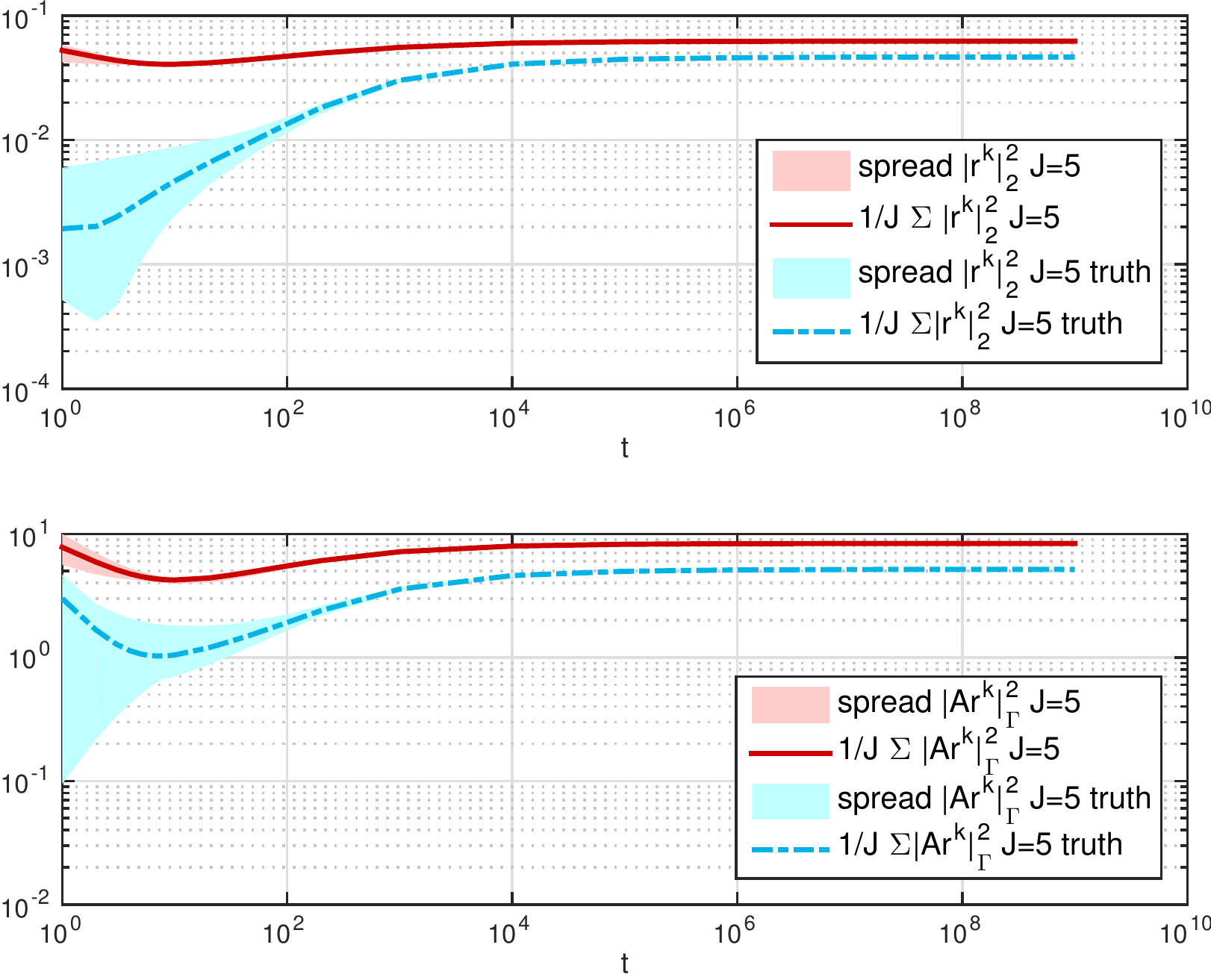}
~\\[-0.75cm]\caption{\footnotesize \label{fig:linNDr}
Quantities $|r|_2^2$, $|Ar|_{\Gamma}^2$ w.r. to $t$, $J=5$ based on KL expansion  of $\cls{C_0}=\beta(A-id)^{-1}$ (red), $J=5$ adaptively chosen (blue), $\beta=10$, $K=2^4-1$, $\eta\sim\cls{\mathcal N}(0,0.01^2 \id)$.}
 \end{figure}
\end{minipage}
 \begin{minipage}{0.025\textwidth}
 ~
 \end{minipage}
 \begin{minipage}{0.45\textwidth}
 \begin{figure}[H]
\centering
    \includegraphics[width=\textwidth]{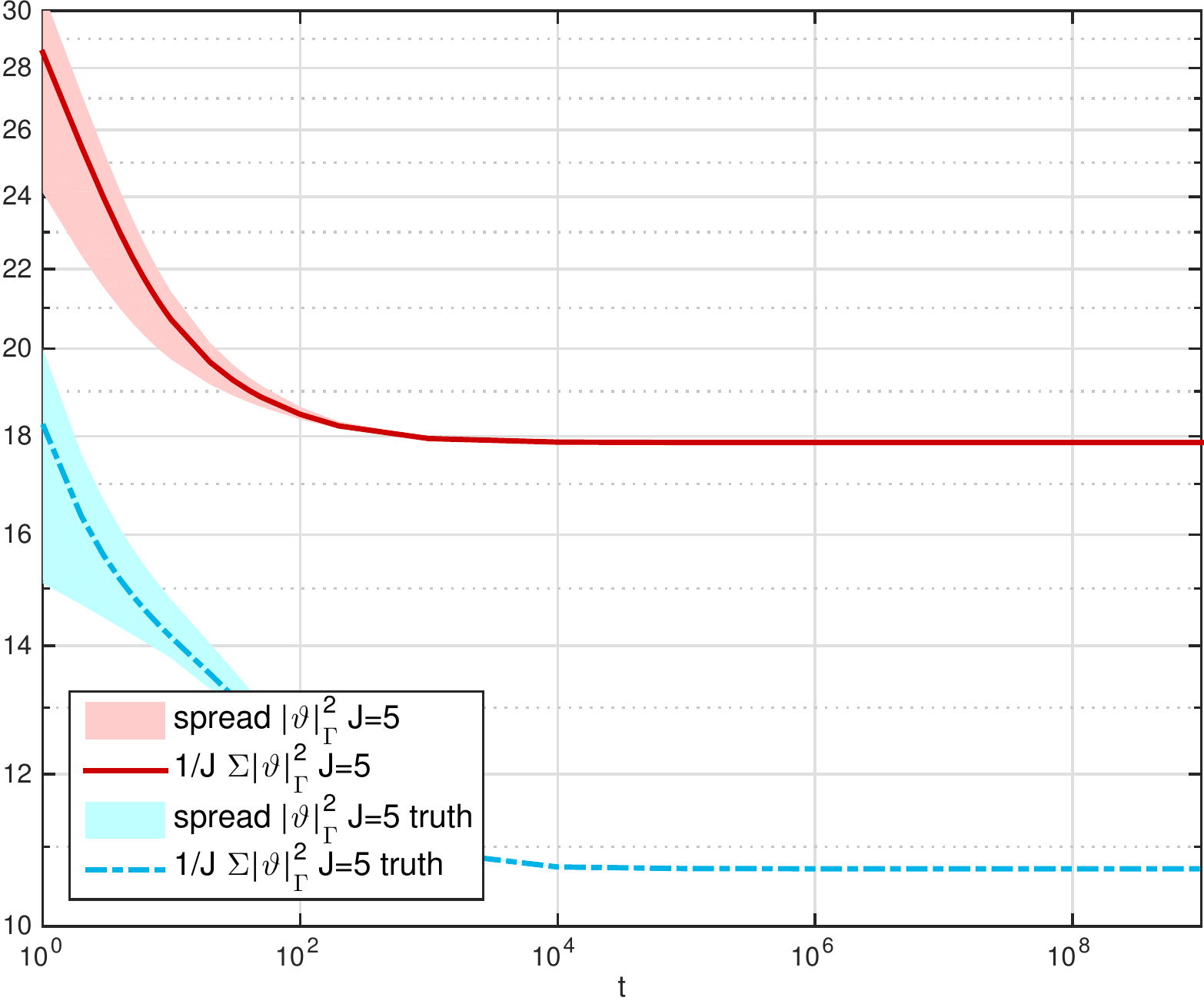}
~\\[-0.75cm]\caption{\footnotesize\label{fig:linNDmis}
Misfit $|\vartheta|_2^2$ w.r. to time $t$, $J=5$ based on KL expansion  of $\cls{C_0}=\beta(A-id)^{-1}$ (red), $J=5$ adaptively chosen (blue), $\beta=10$, $K=2^4-1$, $\eta\sim\cls{\mathcal N}(0,0.01^2 \id)$.}
 \end{figure}
\end{minipage}
~\\[0.2cm]
Figure \ref{fig:linNDr} illustrates the well-known overfitting effect, which arises without using appropriate stopping \cls{criteria}. The method tries to fit the noise in the measurements, which results in an increase in the residuals. This effect is not seen in the misfit functional, cf. Figure \ref{fig:linNDmis} and Figure \ref{fig:linNDmiscon}.

~\\[-0.2cm]
\begin{minipage}{0.035\textwidth}
 ~
 \end{minipage}
\begin{minipage}{0.45\textwidth}
 \begin{figure}[H]
\centering
    \includegraphics[width=\textwidth]{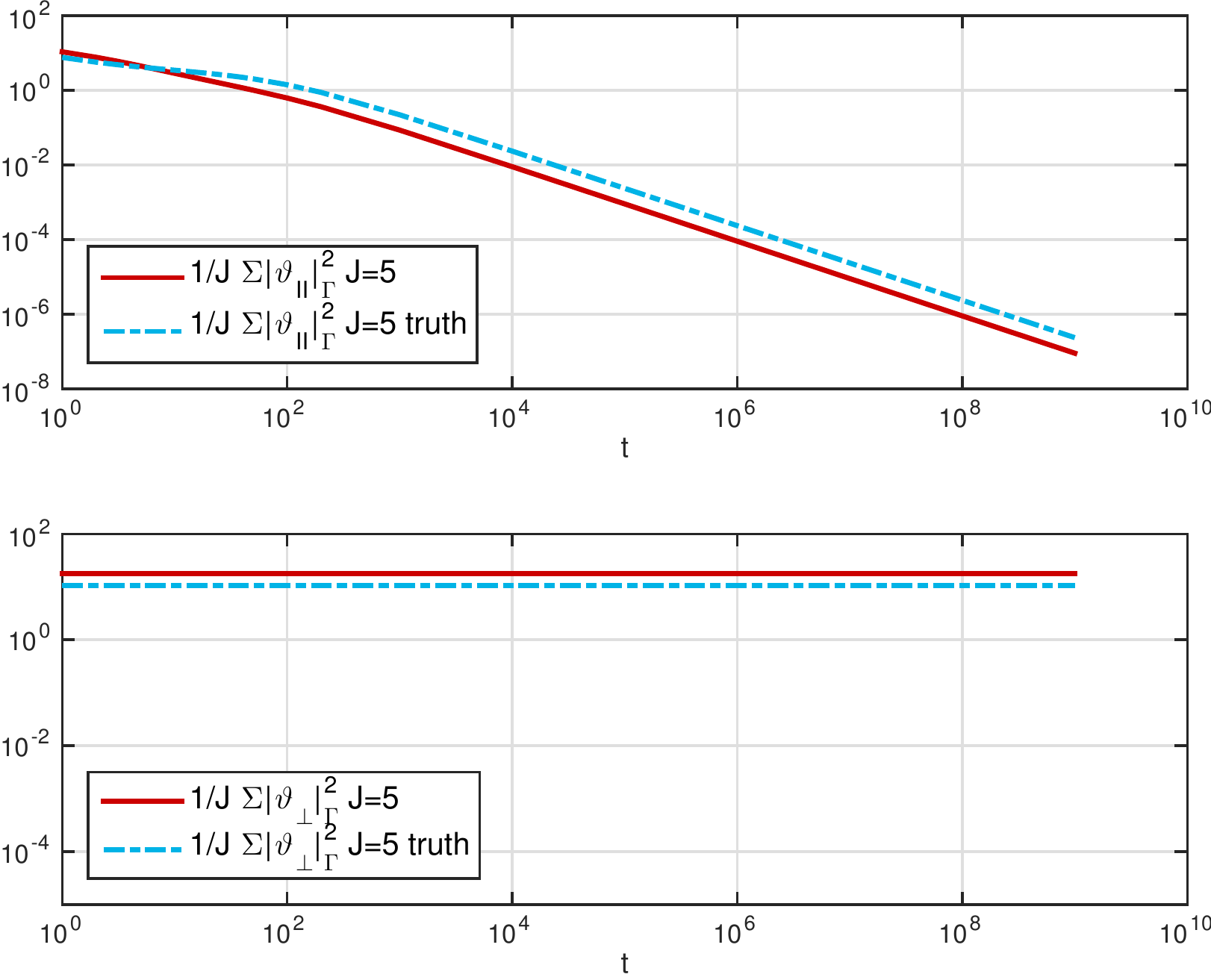}
~\\[-0.75cm]\caption{\footnotesize\label{fig:linNDmiscon}
Quantities $|\dll|_\Gamma^2$, $|\dpp|_{\Gamma}^2$ w.r. to time $t$, $J=5$ based on KL expansion  of $\cls{C_0}=\beta(A-id)^{-1}$ (red), $J=5$ \cls{minimizing the contribution of} $A\rp(t)$ (blue), $\beta=10$, $K=2^4-1$, $\eta\sim\cls{\mathcal N}(0,0.01^2 \id)$.}
 \end{figure}
 \end{minipage}
 \begin{minipage}{0.025\textwidth}
 ~
 \end{minipage}
 \begin{minipage}{0.45\textwidth}
 \begin{figure}[H]
\centering
~\\[0.2cm]
    \includegraphics[width=\textwidth]{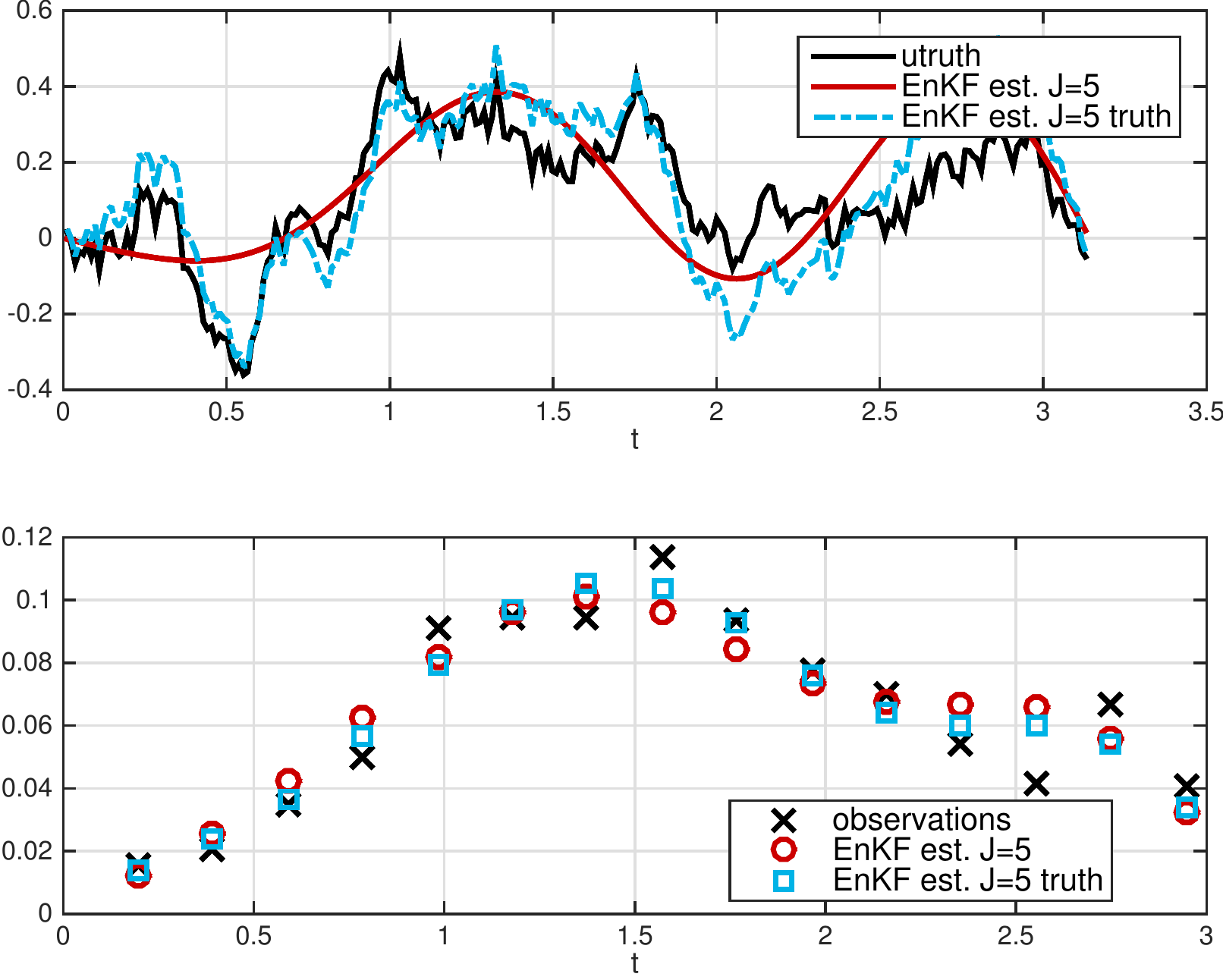}
~\\[-0.75cm]\caption{\footnotesize \label{fig:linNDsol}
Comparison of the EnKF estimate with the truth and the observations, $J=5$ based on KL expansion  of $\cls{C_0}=\beta(A-id)^{-1}$ (red), $J=5$ \cls{minimizing the contribution of} $A\rp(t)$ (blue), $\beta=10$, $K=2^4-1$, $\eta\sim\cls{\mathcal N}(0,0.01^2 \id)$.}
 \end{figure}

\end{minipage}
~\\[0.2cm]
However, the comparison of the EnKF estimates to the truth 
reveals, in Figure \ref{fig:linNDsol}, the strong overfitting effect and suggests the need for a stopping criterion. The Bayesian setting itself provides a so-called a priori stopping rule, i.e. the SMC viewpoint motivates a stopping of the iterations at time $T=1$. Another common choice in the deterministic optimization setting is the discrepancy principle, which accounts for the realization of the noise by checking the following condition $\|\mathcal G(\bar u(t))-y\|_\Gamma\le\tau$, where $\tau>0$ depends on the dimension of the observational space. 
~\hspace*{0.5cm}~

\begin{minipage}{0.45\textwidth}
 \begin{figure}[H]
\centering
~\\[0.15cm]
    \includegraphics[width=\textwidth]{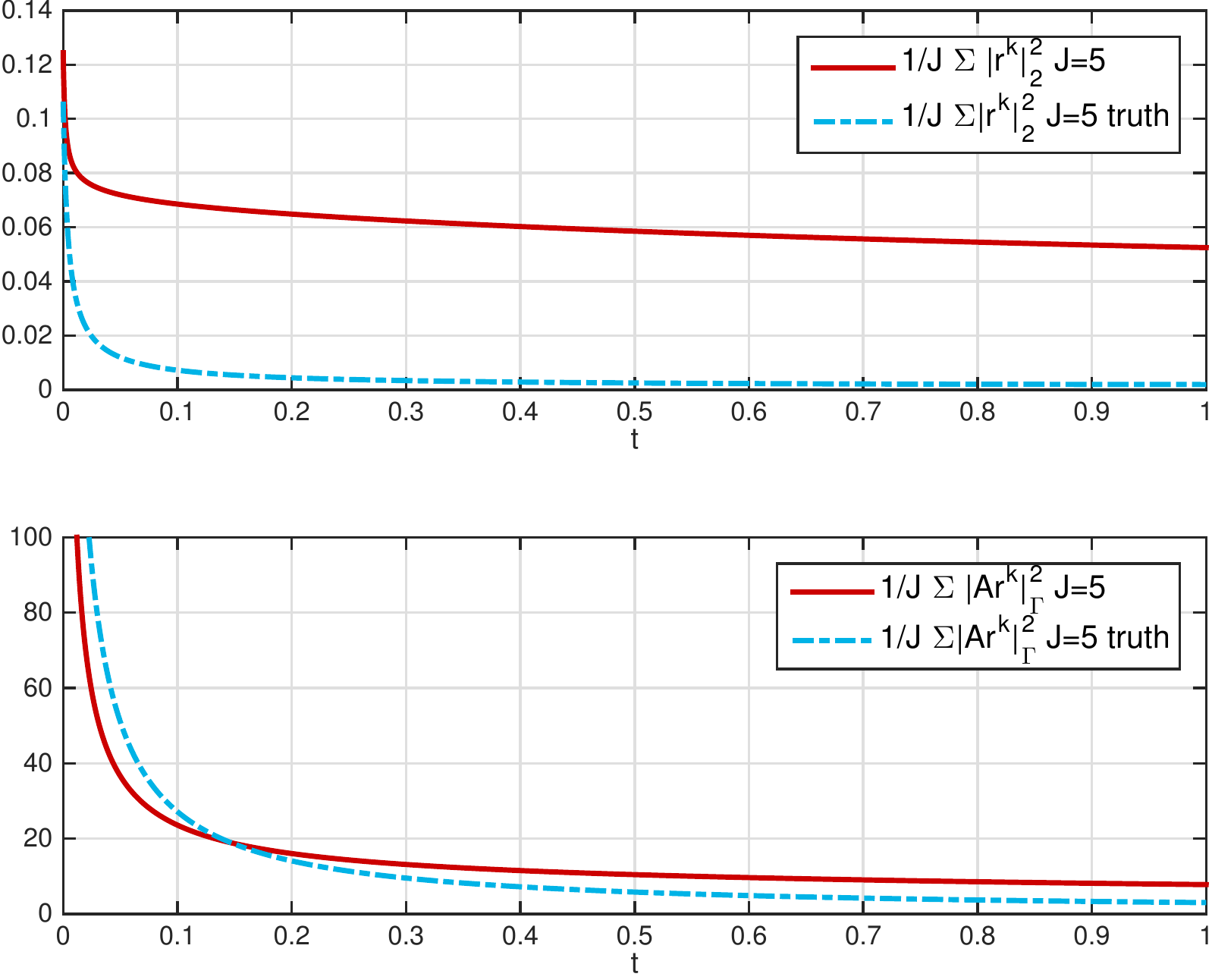}
~\\[-0.75cm]\caption{\footnotesize \label{fig:linNDrT1}
Quantities $|r|_2^2$, $|Ar|_{\Gamma}^2$ w.r. to time $t$, $J=5$ based on KL expansion  of $\cls{C_0}=\beta(A-id)^{-1}$ (red), $J=5$ \cls{minimizing the contribution of} $A\rp(t)$ (blue), $\beta=10$, $K=2^4-1$, $\eta\sim\cls{\mathcal N}(0,0.01^2 \id)$, Bayesian stopping rule.\color{white} $J=5$ {minimizing the contribution of} $A\rp(t)$ (  \color{black}}
 \end{figure}
\end{minipage}
 \begin{minipage}{0.025\textwidth}
 ~
 \end{minipage}
 \begin{minipage}{0.45\textwidth}
 \begin{figure}[H]
\centering
    \includegraphics[width=\textwidth]{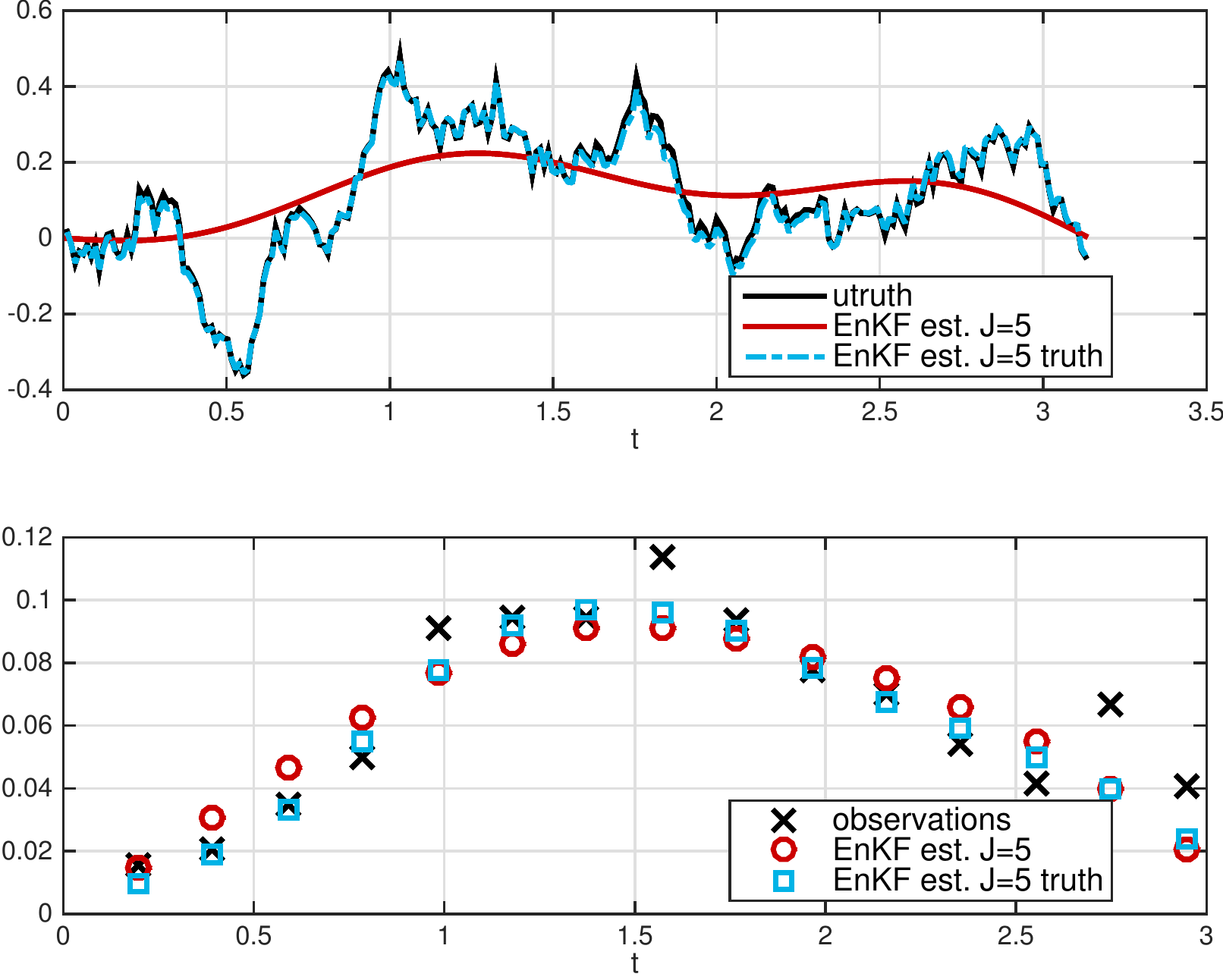}
~\\[-0.75cm]\caption{\footnotesize \label{fig:linNDsolT1}
Comparison of the EnKF estimate with the truth and the observations, $J=5$ based on KL expansion  of $\cls{C_0}=\beta(A-id)^{-1}$ (red), $J=5$ \cls{minimizing the contribution of} $A\rp(t)$ (blue), $\beta=10$, $K=2^4-1$, $\eta\sim\cls{\mathcal N}(0,0.01^2 \id)$, Bayesian stopping rule.}
 \end{figure}
\end{minipage}
~\\[0.2cm]
Figures \ref{fig:linNDrT1} - \ref{fig:linNDsolT1} show the results obtained by employing the Bayesian stopping rule\cls{, i.e. by integrating up to time $T=1$.} The adaptively chosen ensemble leads to much better results in the case of the Bayesian stopping rule. 
 Since we do not expect to have explicit knowledge of the noise, the adaptive 
strategy as presented above is in general not feasible. However an adaptive 
choice of the ensemble according to the misfit may lead to a 
strong overfitting effect, as shown in Figures \ref{fig:linNDmis3} - \ref{fig:linNDsol3}.

~\\[0.2cm]

 \begin{minipage}{0.45\textwidth}
 \begin{figure}[H]
\centering
    \includegraphics[width=\textwidth]{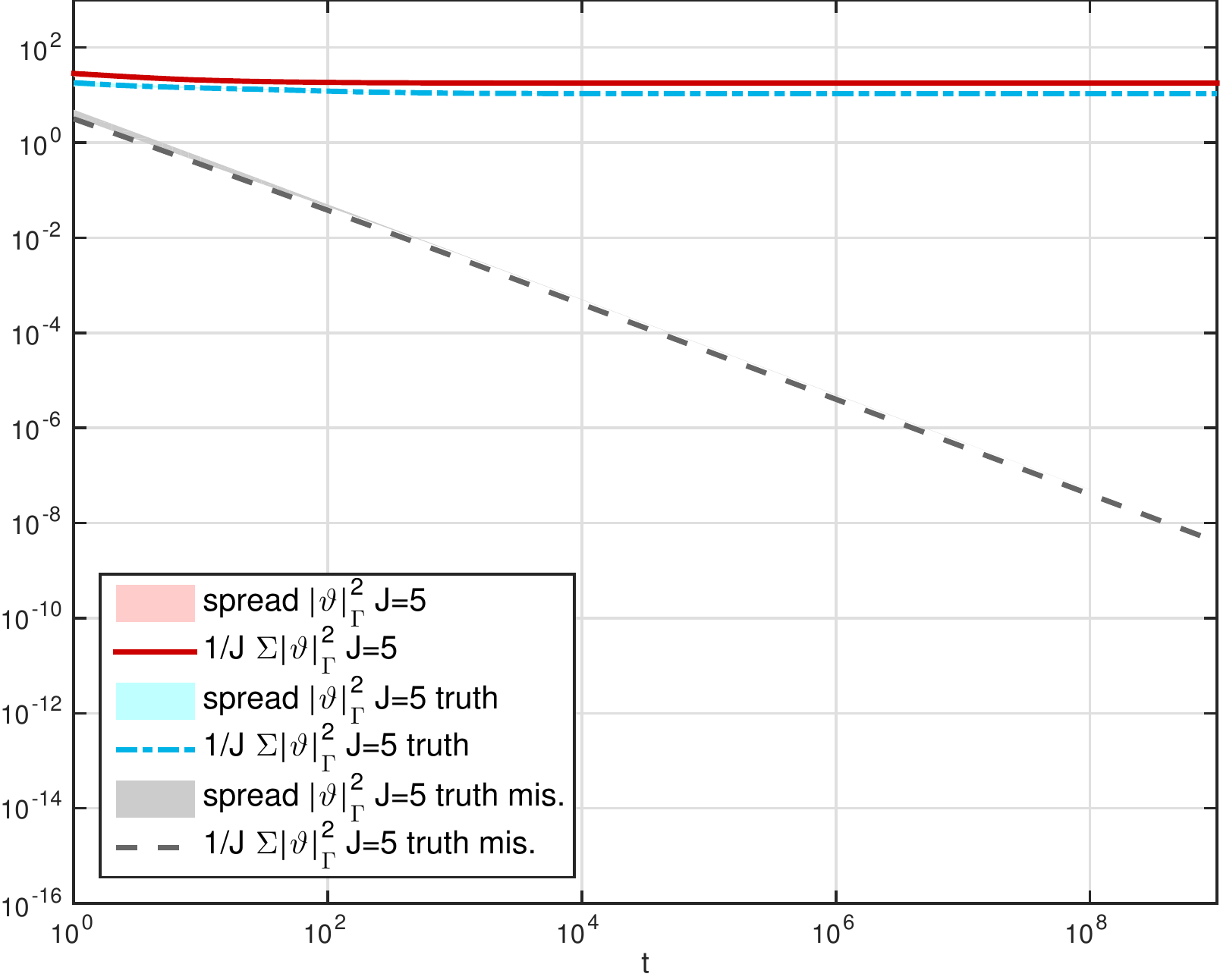}
~\\[-0.75cm]\caption{\footnotesize\label{fig:linNDmis3}
Misfit $|\vartheta|_\Gamma^2$ w.r. to time $t$, $J=5$ based on KL expansion  of $\cls{C_0}=\beta(A-id)^{-1}$ (red), $J=5$ adaptively chosen (blue), $J=5$ \cls{minimizing the contribution of} $\vartheta_\perp$ w.r. to misfit (gray), $\beta=10$, $K=2^4-1$, $\eta\sim\cls{\mathcal N}(0,0.01^2 \id)$.}
 \end{figure}
 \end{minipage}
 \begin{minipage}{0.025\textwidth}
 ~
 \end{minipage}
 \begin{minipage}{0.45\textwidth}
 \begin{figure}[H]
\centering
    \includegraphics[width=\textwidth]{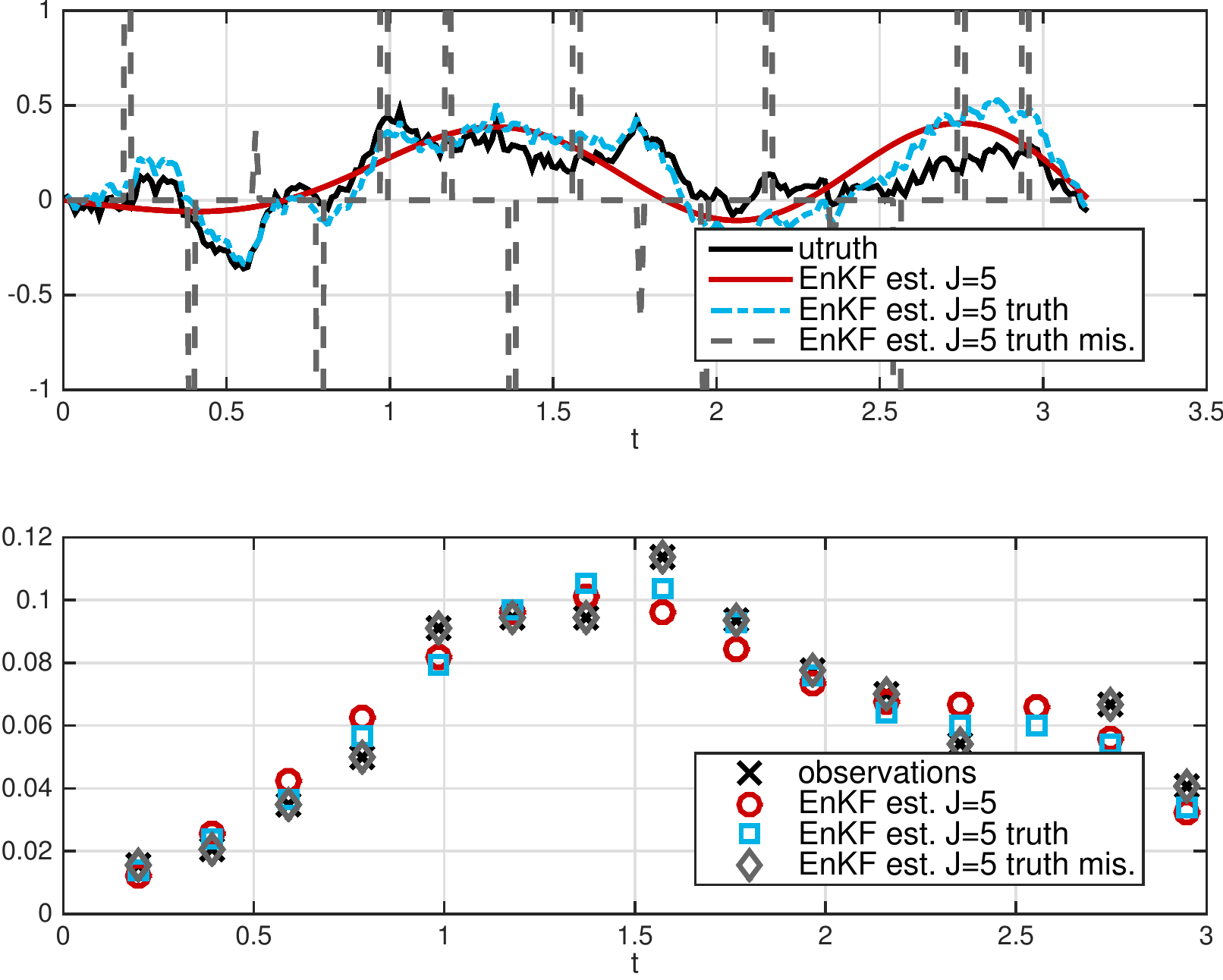}
~\\[-0.75cm]\caption{\footnotesize \label{fig:linNDsol3}
Comparison of the EnKF estimate with the truth and the observations, $J=5$ based on KL expansion  of $\cls{C_0}=\beta(A-id)^{-1}$ (red), $J=5$ adaptively chosen (blue), $J=5$ \cls{minimizing the contribution of} $\vartheta_\perp$ w.r. to misfit (gray), $\beta=10$, $K=2^4-1$, $\eta\sim\cls{\mathcal N}(0,0.01^2 \id)$.}
 \end{figure}

\end{minipage}
~\\[1.6cm]
\cls{
The overfitting effect is still present in the small noise regime as shown below. The standard deviation of the noise is reduced by $100$, i.e. $\eta\sim\cls{\mathcal N}(0,0.001^2 \id)$. 
~\\[0.2cm]

 \begin{minipage}{0.45\textwidth}
 \begin{figure}[H]
\centering
    \includegraphics[width=\textwidth]{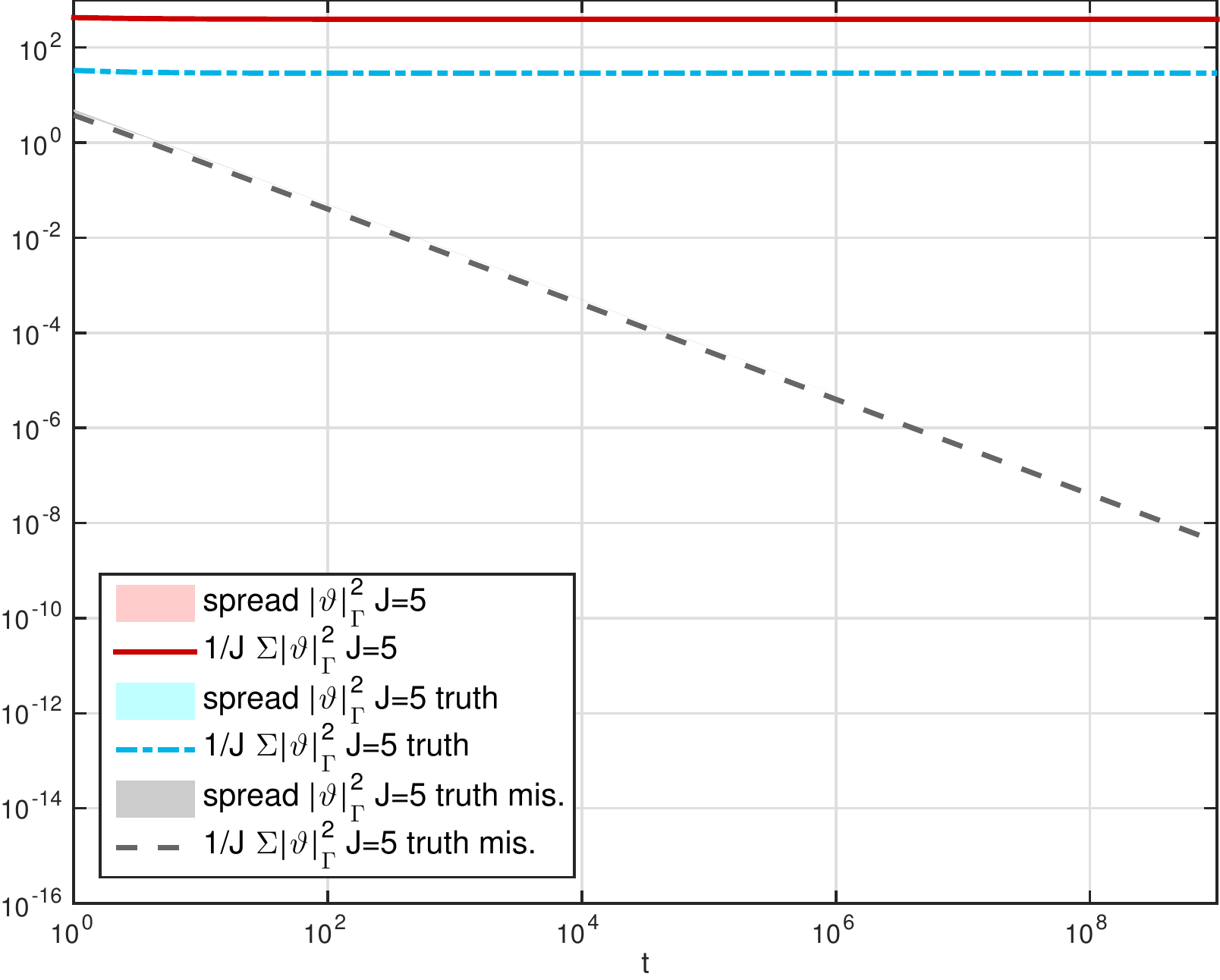}
~\\[-0.75cm]\caption{\footnotesize\label{fig:linNDmis3ns}
Misfit $|\vartheta|_\Gamma^2$ w.r. to time $t$, $J=5$ based on KL expansion  of $\cls{C_0}=\beta(A-id)^{-1}$ (red), $J=5$ adaptively chosen (blue), $J=5$ \cls{minimizing the contribution of} $\vartheta_\perp$ w.r. to misfit (gray), $\beta=10$, $K=2^4-1$, $\eta\sim\cls{\mathcal N}(0,0.001^2 \id)$.}
 \end{figure}
 \end{minipage}
 \begin{minipage}{0.025\textwidth}
 ~
 \end{minipage}
 \begin{minipage}{0.45\textwidth}
 \begin{figure}[H]
\centering
    \includegraphics[width=\textwidth]{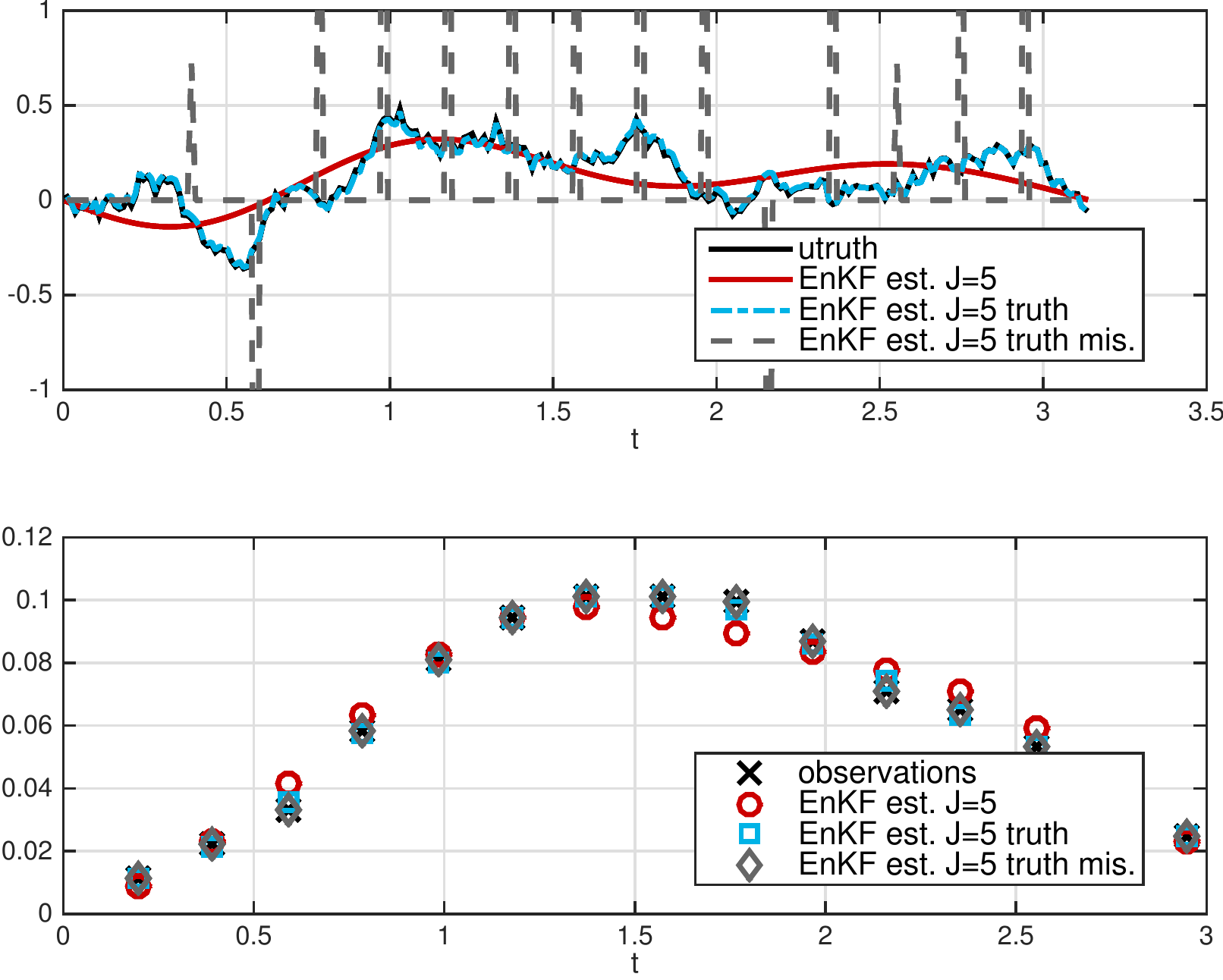}
~\\[-0.75cm]\caption{\footnotesize \label{fig:linNDsol3ns}
Comparison of the EnKF estimate with the truth and the observations, $J=5$ based on KL expansion  of $\cls{C_0}=\beta(A-id)^{-1}$ (red), $J=5$ adaptively chosen (blue), $J=5$ \cls{minimizing the contribution of} $\vartheta_\perp$ w.r. to misfit (gray), $\beta=10$, $K=2^4-1$, $\eta\sim\cls{\mathcal N}(0,0.001^2 \id)$.}
 \end{figure}

\end{minipage}
~\\[0.2cm]
The ill-posedness of the problem leads to the instabilities of the identification problem and requires the use of an appropriate stopping rule.}

\subsection{Nonlinear Forward Model}

To investigate the numerical behavior of the EnKF for nonlinear inverse
problems, we consider the following two-dimensional elliptic PDE: 
 \begin{equation}
   -\diverg(e^u \nabla p)=f \quad \mbox{in } D:=(-1,1)^2\, , \ p=0  \quad \mbox{in } \partial D.
\end{equation}
We aim to find the log permeability $u$ from $49$ observations of the solution
$p$ on a uniform grid in $D$. We choose $f(x)=100$ for the experiments.
The mapping from $u$ to these observations is now nonlinear.
\cls{Again we work in the noise-free case, and take $\Gamma=I$, $\Sigma=0$ and solve \eqref{eq:ode} to estimate the unknown parameters.} The prior is assumed to be Gaussian with covariance operator 
$\cls{C_0}=(-\triangle)^{-2}$, employing homogeneous Dirichlet boundary conditions
to define the inverse of $-\triangle.$ We use a FEM approximation
based on continuous, piecewise linear ansatz functions on a uniform mesh with meshwidth $h=2^{-4}$.
The initial ensemble of size $5$ and $50$ 
is chosen based on the KL expansion of $\cls{C_0}$ in the same way as in the
previous subsection.

The results given in Figure \ref{fig:nonNFe} and Figure \ref{fig:nonNFr} show a similar behavior as in the linear case. 
The approximation quality of the subspace spanned by the initial ensemble clearly influences, also in the nonlinear example, the accuracy of the estimate.
Taking a look at the EnKF estimate in this nonlinear setting, we observe a satisfactory approximation of the truth and a perfect match of the observational data in the case of the larger ensemble, cf. Figure \ref{fig:nonrescomp2}.

~\\[0.1cm]

\begin{minipage}{0.45\textwidth}
 \begin{figure}[H]
\centering
    \includegraphics[width=\textwidth]{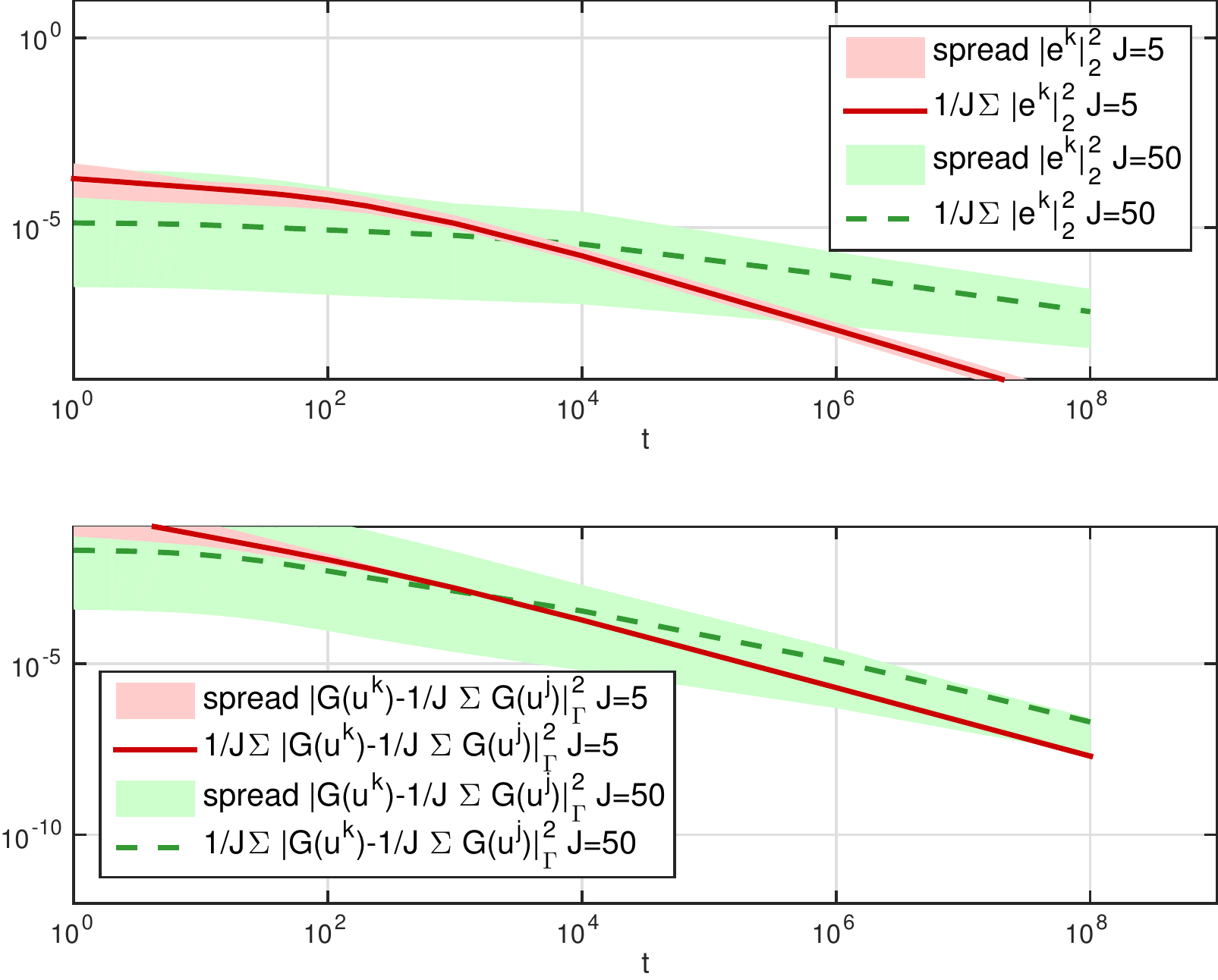}
~\\[-0.75cm]\caption{\footnotesize \label{fig:nonNFe}
Quantities $|e^{(k)}|_2^2$, $|\mathcal G(u^{(k)})-\frac1J\sum_{j=1}^J\mathcal G(u^{(j)}|_{\Gamma}^2$ w.r. to time $t$, $J=5$ (red) and $J=50$ (green), initial ensemble chosen based on KL expansion of $\cls{C_0}=(-\triangle)^{-2}$. }
 \end{figure}
\end{minipage}
 \begin{minipage}{0.025\textwidth}
 ~
 \end{minipage}
 \begin{minipage}{0.45\textwidth}
 \begin{figure}[H]
\centering
    \includegraphics[width=\textwidth]{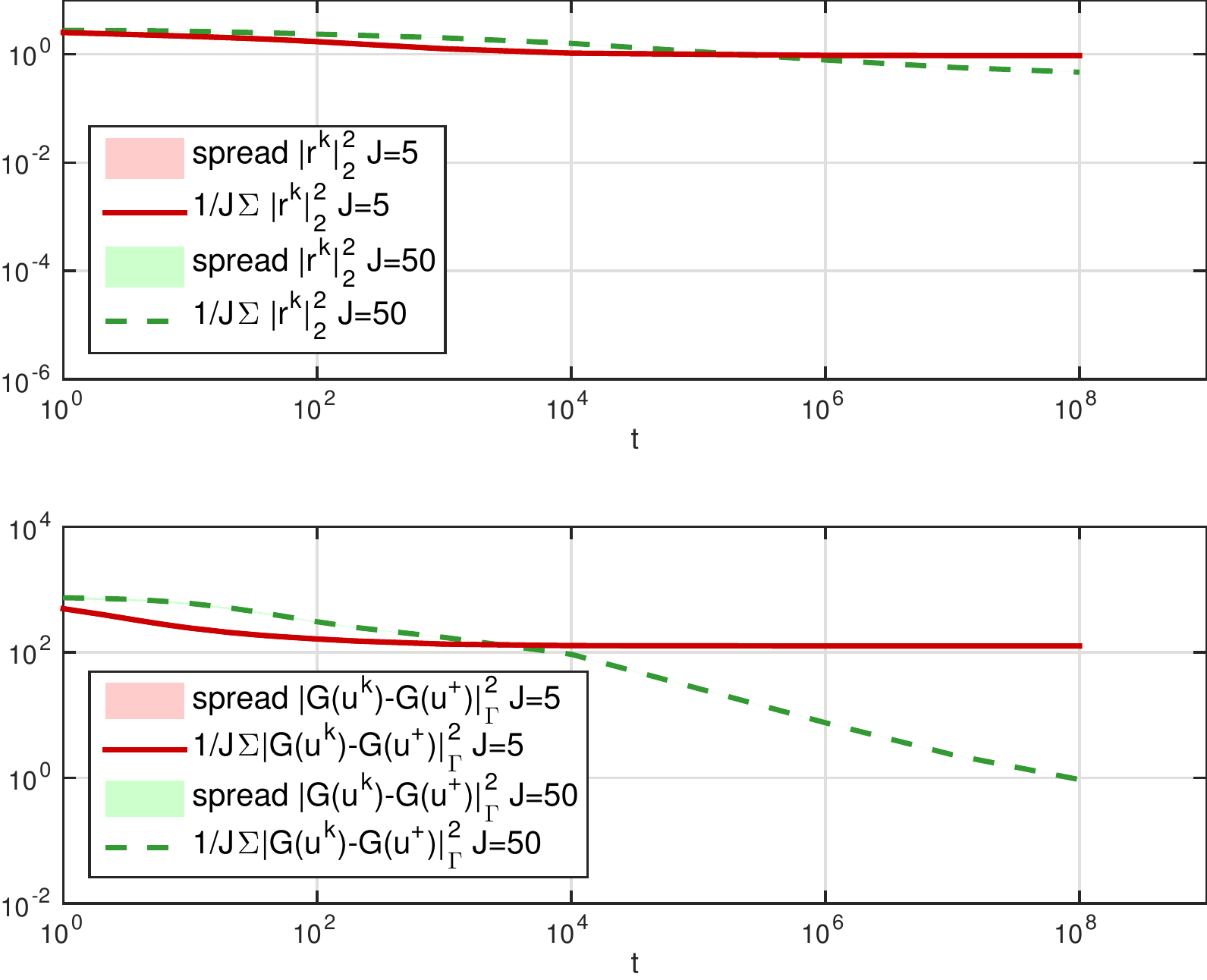}
~\\[-0.75cm]\caption{\footnotesize\label{fig:nonNFr}
Quantities $|r^{(k)}|_2^2$  $|\mathcal G(u^{(k)})-\mathcal G(u^\dagger) |_{\Gamma}^2$ w.r. to time $t$, $J=5$ (red) and $J=50$ (green), initial ensemble chosen based on KL expansion of $\cls{C_0}=(-\triangle)^{-2}$. }
 \end{figure}
\end{minipage}

~\\[-0.1cm]
\begin{figure}[H]
\centering

    \includegraphics[width=0.3\textwidth]{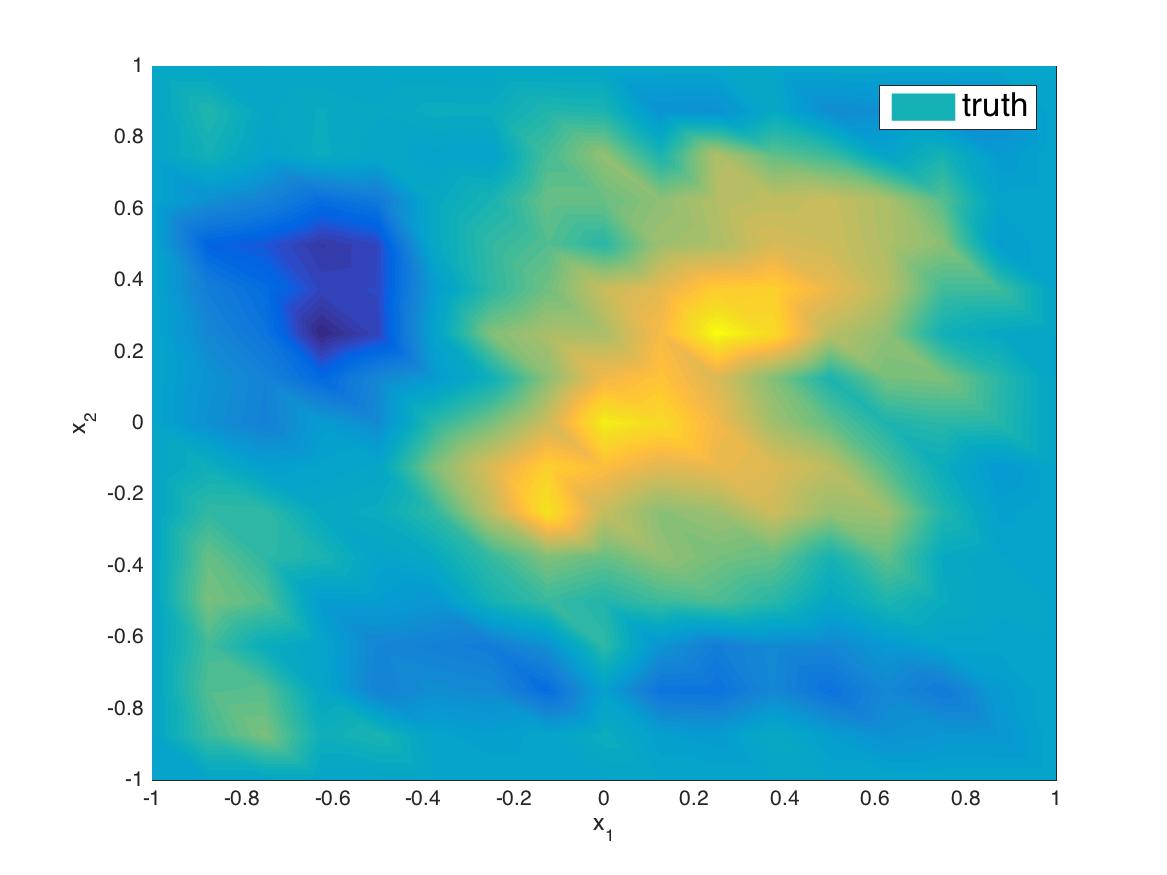}
\hspace{0.0cm}
 \includegraphics[width=0.3\textwidth]{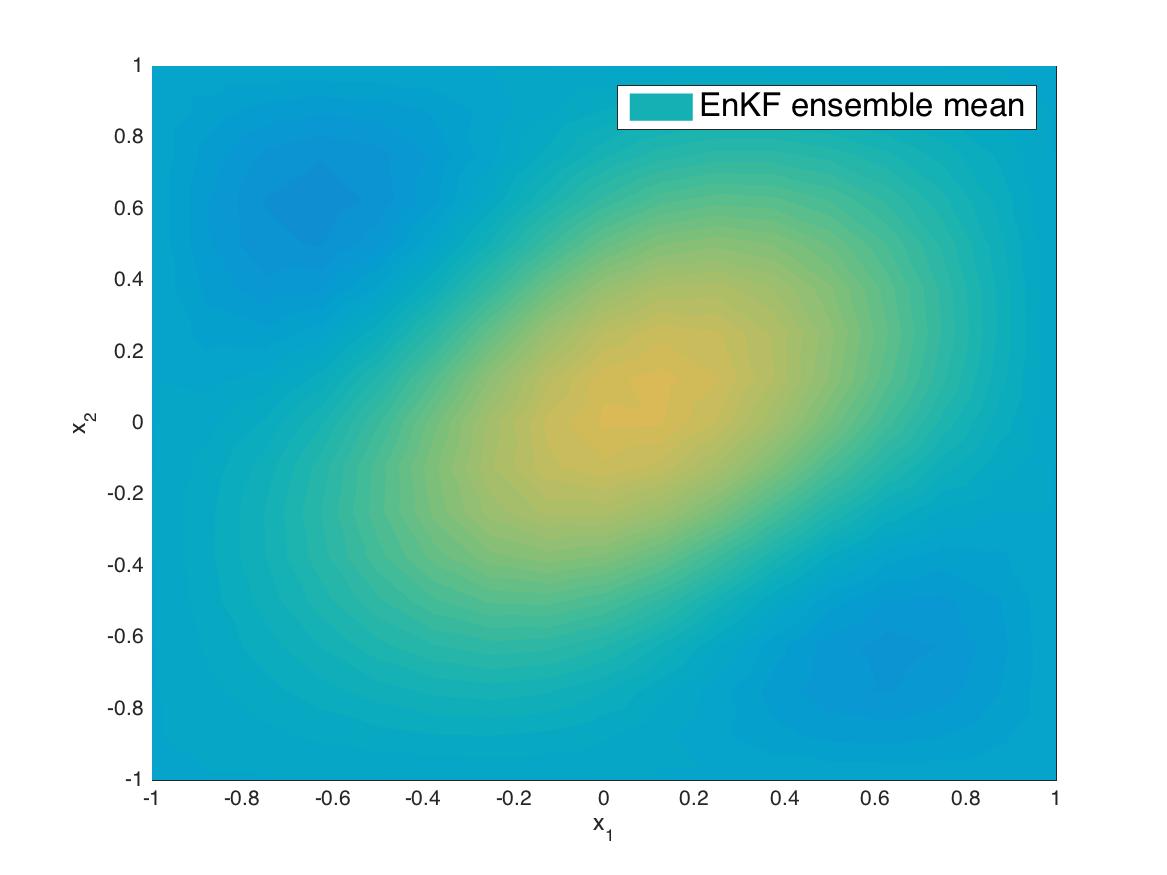}
 \hspace{0.0cm}
 \includegraphics[width=0.3\textwidth]{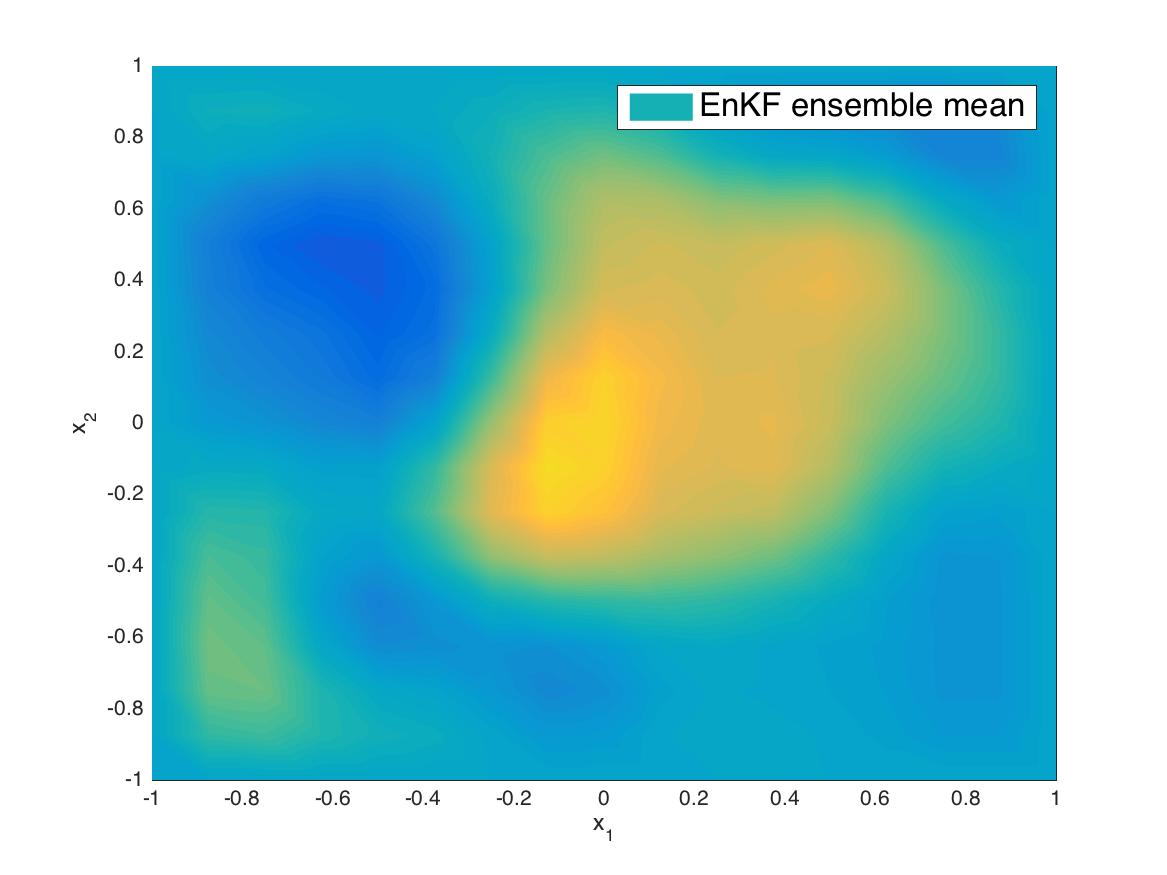}
   \includegraphics[width=0.3\textwidth]{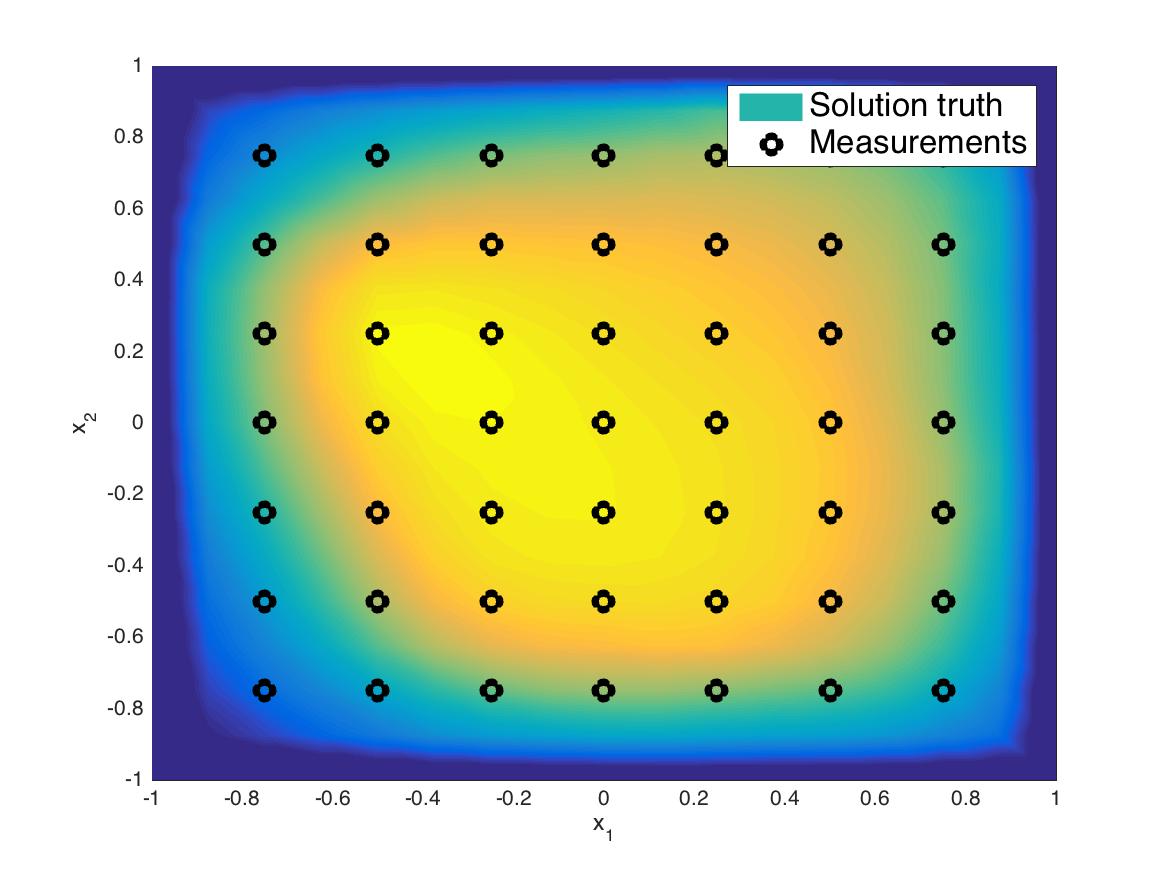}
\hspace{0.0cm}
    \includegraphics[width=0.3\textwidth]{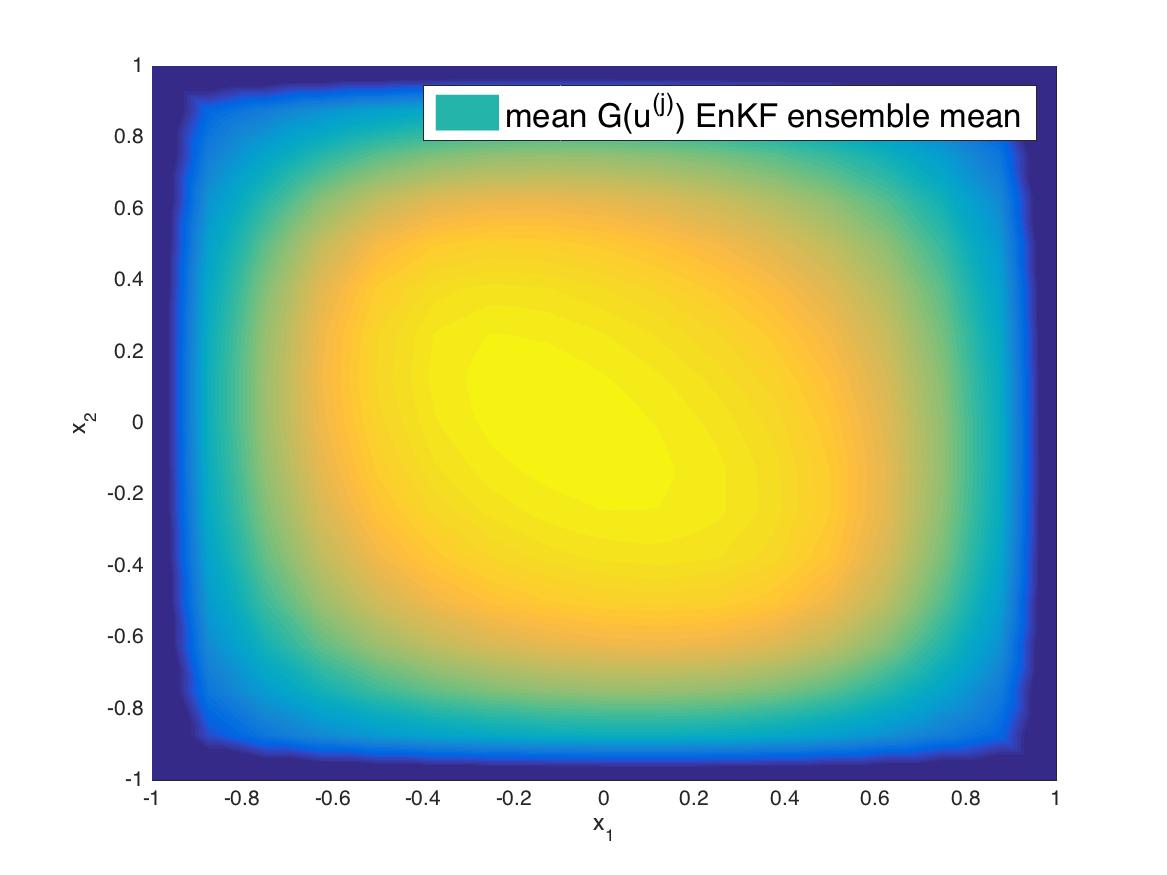}
\hspace{0.0cm}
    \includegraphics[width=0.3\textwidth]{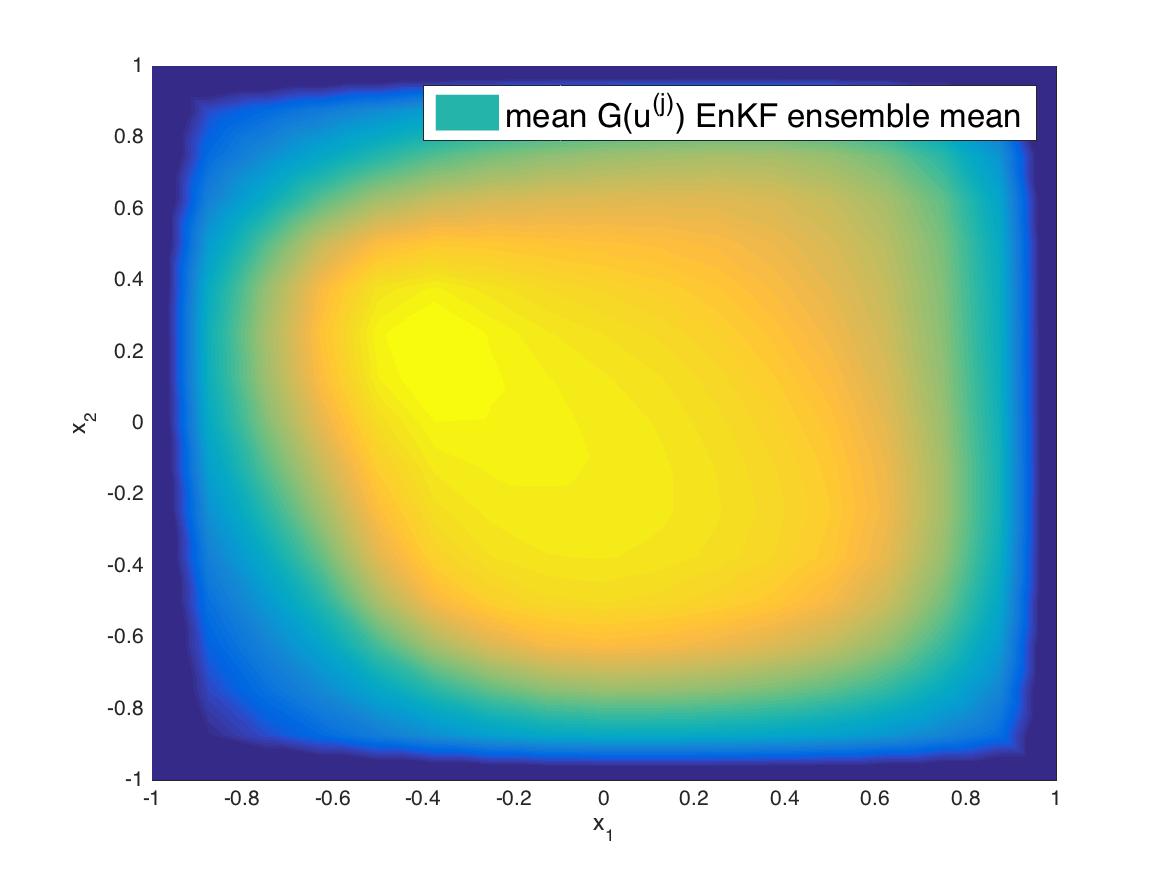}
{\footnotesize
~\\[-0.25cm]\caption{\footnotesize  \label{fig:nonrescomp2} Comparison of the truth (left above) and the EnKF estimate w.r. to $x$, J=5 (middle above) , $J=50$ (right above) and comparison of the forward solution $G(u^\dagger)$ (left below) and the estimated solutions of the forward problem $J=5$ (middle below), $J=50$ (right below). }}
\end{figure}
~\\[0.0cm]

\section{Variants on EnKF}
\label{sec:CV}

In this section we describe three variants on the EnKF, all formulated
in continuous time in order to facilitate comparison with the preceding
studies of the standard EnKF in continuous time. The first two
methods, variance inflation and localization, are commonly used
by practitioners in both the filtering and inverse problem
scenarios \cls{\cite{Sheikh,TELA:TELA216}}.
The third method, random search, is motivated by the
SMC derivation of the EnKF for inverse problems, and is new in the
context of the EnKF. For all three methods we provide numerical
results which illustrate the behavior of the EnKF variant, in
comparison with the standard method. \cls{In the following, we focus on the linear case with $\Sigma=0$. The methods from the first two subsections have
generalizations in the general nonlinear setting, and indeed are widely used
in data assimilation and, to some extent, in geophysical inverse problems;
but we present them in a form tied to the equation \eqref{eq:lode2} which is
derived in the linear case. The method in the final subsection is implemented
through an entirely derivative-free MCMC method, and is hence automatically
defined, as is, for nonlinear as well as linear inverse problems.} 

{
\subsection{Variance Inflation}

The empirical covariances $C^{up}$ and $C^{pp}$ all have rank no
greater than $J-1$ and hence are rank deficient whenever the number of
particles $J$ is less than the dimension of the space $X$.
Variance inflation proceeds by correcting such rank deficiencies by the addition of self-adjoint, strictly positive
operators.  A natural variance inflation
technique is to add a multiple of the prior covariance $C_0$ to the empirical
covariance which gives rise to the equations 
\begin{equation}\label{eq:uVI}
\frac{\dd u^{(j)}}{\dd t}=-\bigl(\alpha C_0+C(u)\bigr) D_u\Phi(u^{(j)};y),
\quad j=1,\dots, J\,,
\end{equation} 
where $\Phi$ is as defined in \eqref{eq:fil}.
Taking the inner-product in $X$ with $D_u\Phi(u^{(j)};y)$ we deduce that
\begin{equation}
\label{eq:ap2}
\frac{\dd \Phi(u^{(j)};y)}{\dd t} \le -\alpha \|C_0^{\frac12} D_u\Phi(u^{(j)};y)\|^2\;,
\end{equation}
This implies that all $\omega-$limit points of the dynamics are contained
in the critical points of $\Phi(\cdot;y).$

\subsection{Localization}

Localization techniques aim to remove spurious long distance correlations by modifying the covariance operators $C^{up}$ and $C^{pp}$, or directly the Kalman gain $C^{up}_{n+1}(C^{pp}_{n+1}+h^{-1}\Gamma)^{-1}$.
 Typical convolution kernels reducing the influence of distant regions are of the form
\begin{equation}
\begin{aligned}
\label{eq:add}
&\rho:D \times D \rightarrow \mathbb R\\
&\rho(x,y)=\exp(-|x-y|^r)\;,
\end{aligned}
\end{equation}
where $D\subset \mathbb R^d, \ d\in \mathbb N$ denotes the physical domain and $| \cdot |$ is a suitable norm in $D$,  $r \in \mathbb N$, cf. \cite{KLS}. The continuous time limit in the linear setting then reads as
 \begin{equation}
\label{eq:loc}
\frac{\dd u^{(j)}}{\dd t}=-C^{\rm loc}(u)  D_u\Phi(u^{(j)};y),
\quad j=1,\dots, J\,,
\end{equation} 
where $C^{\rm loc}(u) \phi(x)=\int_D \phi(y) k(x,y) \rho(x,y)\dd y\;$ with $k$ denoting the kernel of $C(u)$ and $\phi \in \cX$.

\subsection{Randomized Search}

We notice that the mapping on probability measures given by 
\eqref{eq:smc1} may be replaced by
\begin{equation}
\label{eq:smc2}
  {\mu_{n+1}=L_nP_n\mu_n}\; .
  \end{equation}
where $P_n$ is any Markov kernel which preserves $\mu_n.$ 
For example we may take $P_n$ to be the pCN method \cite{cotter2013mcmc} for 
measure $\mu_n$. One step of the pCN method for given particle $u^{(j)}_n$ 
in iteration $n$ is realized by
\begin{itemize}
\item Propose $v^{(j)}_n=\sqrt{(1-\beta^2)}u^{(j)}_n+\beta \iota^{(j)}$, $\iota^{(j)}\sim\mathcal N(0,C_0)$.
\item Set $\tilde u^{(j)}_{n}=v^{(j)}_n$  with probability $a(u^{(j)}_n, v^{(j)}_n)$.
\item Set $\tilde u^{(j)}_{n}=u^{(j)}_n$ otherwise
\end{itemize}
\cls{assuming the prior is Gaussian, i.e. $\mathcal N(0,C_0)$.}
The acceptance probability is given by 
$$a(u^{(j)}_n, v^{(j)}_n)=\min\{1,\exp(nh\Phi(u^{(j)}_n)-nh\Phi(v^{(j)}_n))\}.$$ 
The particles $\tilde u^{(j)}_{n}$ are used to approximate the measure $\tilde \mu_{n}=P_n\mu_n$, which is then mapped to $\mu_{n+1}$ by the application of Bayes' theorem, i.e. $\mu_{n+1}=L_n\tilde\mu_{n}$. 

\cls{Using the continuous-time diffusion limit arguments from \cite[Theorem 4]{PST11}, 
which apply in the nonlinear case, and combining with the continuous time
limits described for the EnKF earlier in this paper, we obtain}
\begin{equation}
\label{eq:dl}
\begin{split}
\frac{\dd u^{(j)}}{\dd t}=
\frac{1}{J}\sum_{k=1}^J \bigl\langle \cG(u^{(k)})-\bG,
y-\cG(u^{(j)})\bigr\rangle_{\Gamma}
\bigl(u^{(k)}-\bu\bigr)
\\
-u^{(j)}-t C_0 D_u\Phi(u^{(j)};y)+\sqrt{2C_0}\frac{\dd W^{(j)}}{\dd t}\;.
\end{split}
\end{equation}
\cls{Although the limiting equation involves gradients of $\Phi$, and hence
adjoints for the forward model, the discrete time implementation above
avoids the gradient computation by using the accept-reject step, and remains
a derivative free optimizer.}

\subsection{Numerical Results}

In the following, to illustrate behavior of the EnKF variants, we present 
numerical experiments for the linear forward problem in the noise-free case:
\eqref{eq:forward} and \eqref{eq:modelne} with $\eta=0$. The performance of the EnKF variants is compared to the basic algorithms shown in Figure \ref{fig:linNFr} and Figure \ref{fig:linNFsol}.

\subsubsection{Inflation}
We investigate the numerical behavior of variance inflation of the form
given in \eqref{eq:uVI} with $\alpha=0.01$.
Figures \ref{fig:linNFrinfl} and \ref{fig:linNFsolinfl} show that the variance inflated method becomes a preconditioned gradient flow, which, in the linear case, leads to fast convergence of the projected iterates.
It is noteworthy that in this case there is very little difference in
behavior between ensemble sizes of $5$ and $50$.

~\\[0.1cm]

 \begin{minipage}{0.45\textwidth}
 \begin{figure}[H]
\centering
    \includegraphics[width=\textwidth]{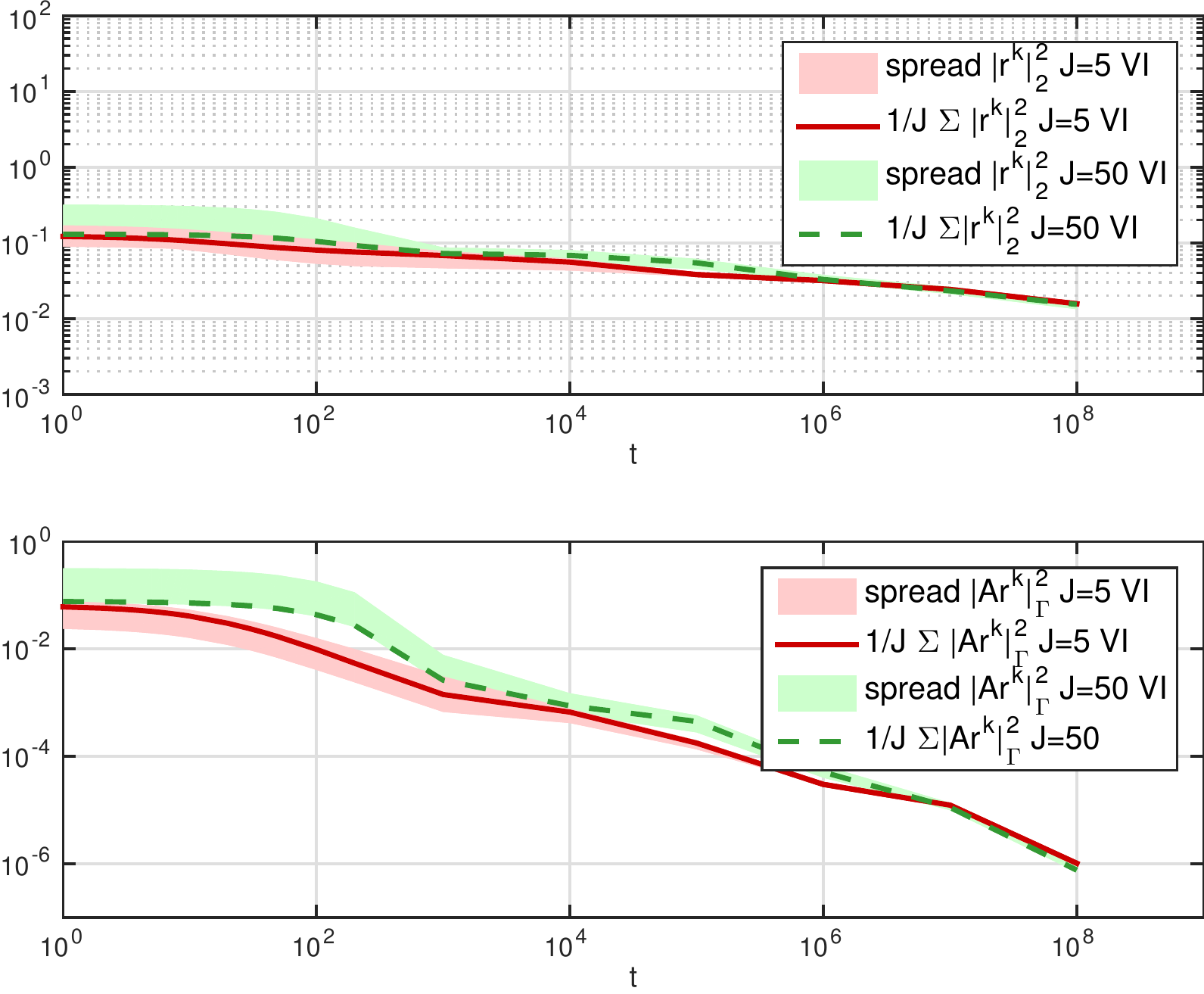}
~\\[-0.75cm]\caption{\footnotesize\label{fig:linNFrinfl}
Quantities $|r|_2^2$, $|Ar|_{\Gamma}^2$ w.r. to time $t$, $J=5$ with variance inflation (red) and $J=50$ with variance inflation (green),  $\beta=10$, $K=2^4-1$, initial ensemble chosen based on KL expansion of $\cls{C_0}=\beta(A-id)^{-1}$. \color{white}{Comparison of the EnKF }\color{black}}
 \end{figure}
\end{minipage}
 \begin{minipage}{0.025\textwidth}
 ~
 \end{minipage}
\begin{minipage}{0.45\textwidth}
 \begin{figure}[H]
\centering
    \includegraphics[width=\textwidth]{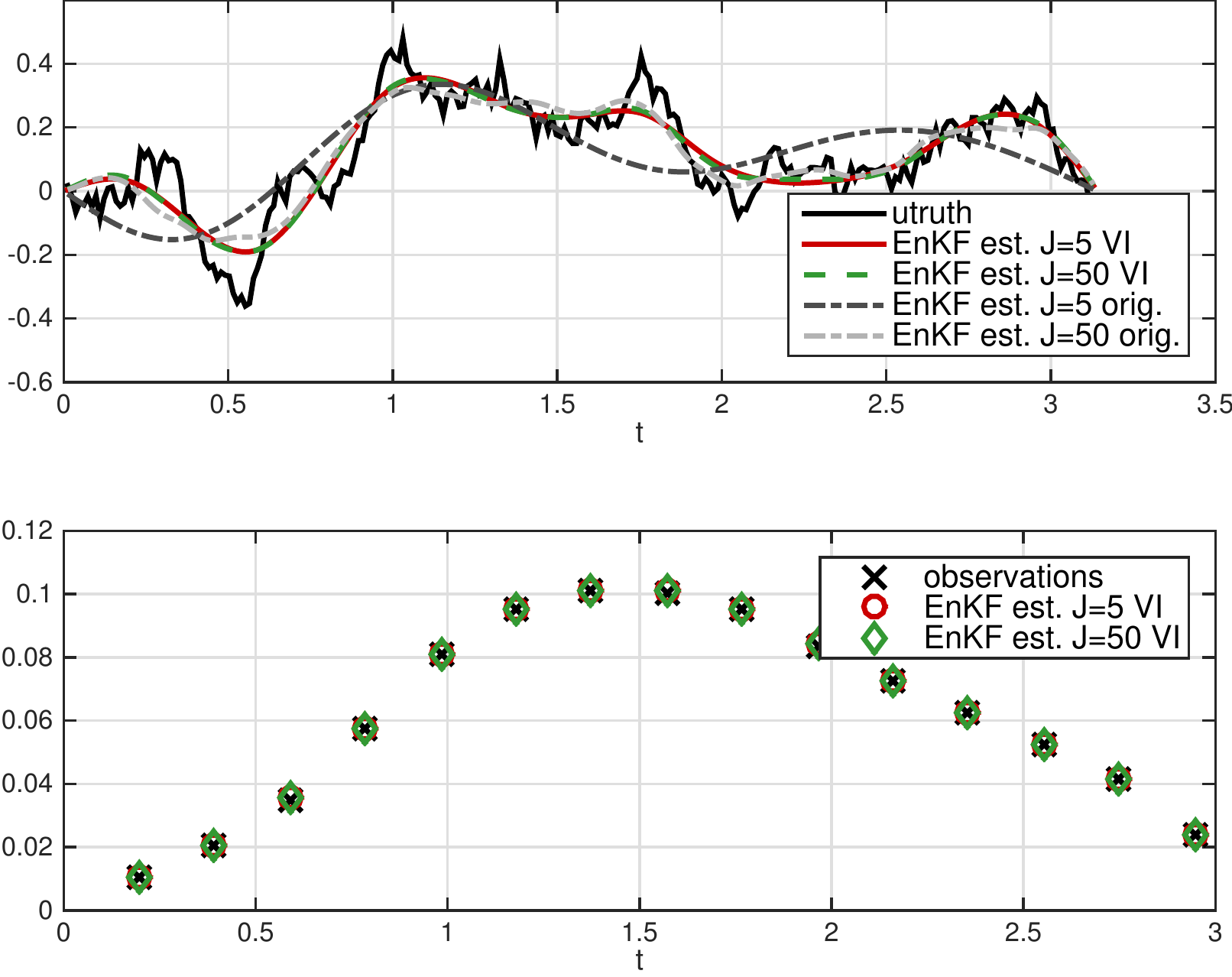}
~\\[-0.75cm]\caption{\footnotesize \label{fig:linNFsolinfl}
Comparison of the EnKF estimate with the truth and the observations, $J=5$ with variance inflation (red) and $J=50$ with variance inflation (green),  $\beta=10$, $K=2^4-1$, initial ensemble chosen based on KL expansion of $\cls{C_0}=\beta(A-id)^{-1}$. }
 \end{figure}
\end{minipage}
~\\[0.2cm]

\subsubsection{Localization}
We consider a localization of the form given 
by equations \eqref{eq:add}, \eqref{eq:loc} with 
\cls{$r=2$} and Euclidean norm inside the cut-off kernel. 
Figures \ref{fig:linNFrloc} - \ref{fig:linNFsolloc} clearly demonstrate the improvement by the localization technique, which can overcome the linear span property and thus, leads to better estimates of the truth.

 ~\\[0.1cm]
 
 \begin{minipage}{0.45\textwidth}
 \begin{figure}[H]
\centering
    \includegraphics[width=\textwidth]{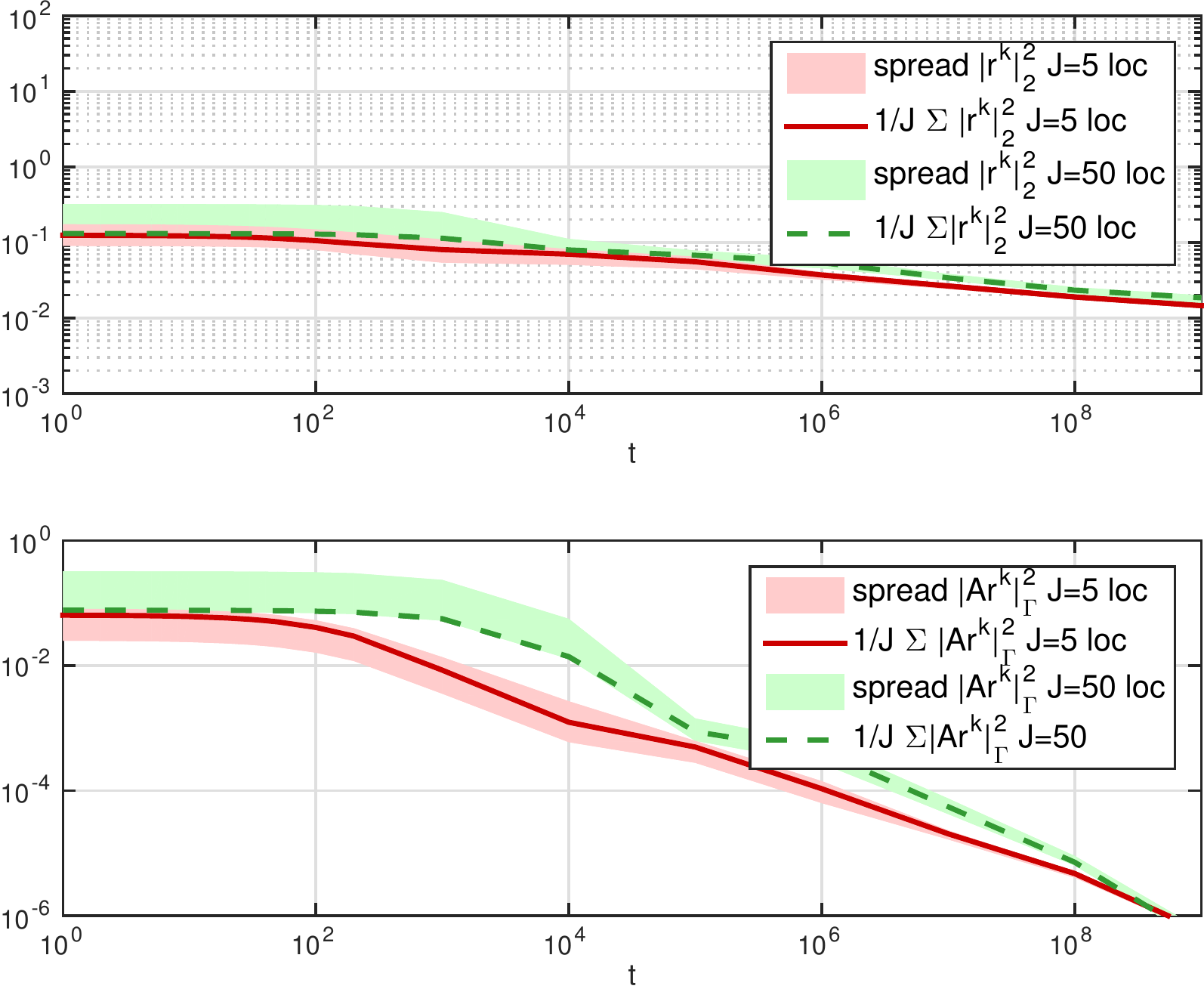}
~\\[-0.75cm]\caption{\footnotesize\label{fig:linNFrloc}
Quantities $|r|_2^2$, $|Ar|_{\Gamma}^2$ w.r. to time $t$, $J=5$ with localization (red) and $J=50$ with localization (green), $\beta=10$, $K=2^4-1$, initial ensemble chosen based on KL expansion of $\cls{C_0}=\beta(A-id)^{-1}$. \color{white}{Comparison of the EnKF }\color{black}}
 \end{figure}
\end{minipage}
 \begin{minipage}{0.025\textwidth}
 ~
 \end{minipage}
 \begin{minipage}{0.45\textwidth}
 \begin{figure}[H]
\centering
    \includegraphics[width=\textwidth]{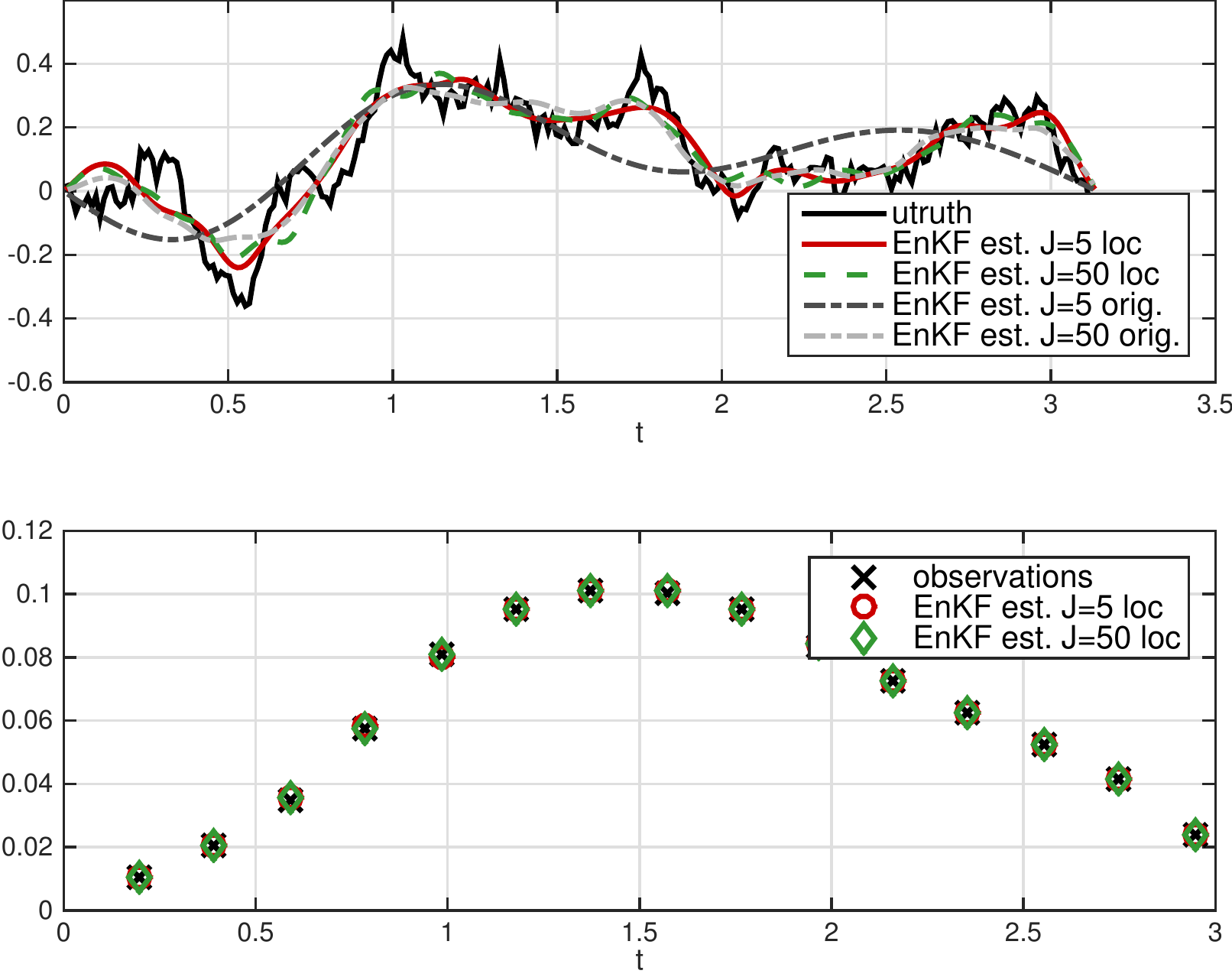}
~\\[-0.75cm]\caption{\footnotesize \label{fig:linNFsolloc}
Comparison of the EnKF estimate with the truth and the observations, $J=5$ with localization (red) and $J=50$ with localization (green),  $\beta=10$, $K=2^4-1$, initial ensemble chosen based on KL expansion of $\cls{C_0}=\beta(A-id)^{-1}$. }
 \end{figure}
\end{minipage}

\subsubsection{Randomized Search}

We investigate the behavior of randomized search for the 
linear problem with $\mathcal G(\cdot)=A\cdot$. 
For the numerical solution of the continuous limit \eqref{eq:dl},
we employ a splitting scheme with a linearly implicit Euler step, 
namely
\begin{eqnarray*}
\tilde u_{n+1}^{(j)}&=&\sqrt{1-2h}u_{n}^{(j)}+\sqrt{2h C_0} \zeta_n\\
Ku_{n+1}^{(j)}&=&\tilde u_{n+1}^{(j)}
+h(C(\tilde u_{n+1})A^*\Gamma^{-1}y^\dagger+nhC_0 A^*\Gamma^{-1}y^\dagger)\;,
\end{eqnarray*}
where $\zeta_n \sim N(0,id)$ and
\cls{$K:=I+h(C(\tilde u_{n+1}) A^*\Gamma^{-1} A+nhC_0 A^*\Gamma^{-1} A).$}
In all numerical experiments reported we take $h=2^{-8}$.
Figure \ref{fig:linNFrrsMM} and Figure \ref{fig:linNFsolrsMM} show that the randomized search leads to an improved performance compared to the \cls{original} EnKF method. Due to the fixed step size and the resulting high computational costs, the solution is computed up to time \cls{$T=100$}. \cls{In order to accelerate the numerical solution of the limit \eqref{eq:dl}, implicit schemes can be considered. Note that the limit requires the computation of the gradients, which is in practice undesirable. However, the limit reveals from a theoretical point of view important structure, whereas the discrete version is more suitable for applications.} The advantage of the randomized search is
apparent.

\subsubsection{Summary} 
The experiments show a similar performance for all discussed variants. The variance inflation technique and the localization variant both lead to gradient flows, which are, in the noise-free case, favorable due to the fast convergence. \cls{On the other hand, 
these strategies also accelerate the convergence of the ensemble to the mean 
(ensemble collapse) and may be considered less desirable for this reason. 
The randomized search preserves by construction the spread of the ensemble. A similar regularization effect is achieved by perturbing the observational data, see \eqref{eq:ode}. The variants all break the subspace property of the original version, which results in an improvement in the estimate.}

~\\[0.1cm]

 \begin{minipage}{0.45\textwidth}
 \begin{figure}[H]
\centering
    \includegraphics[width=\textwidth]{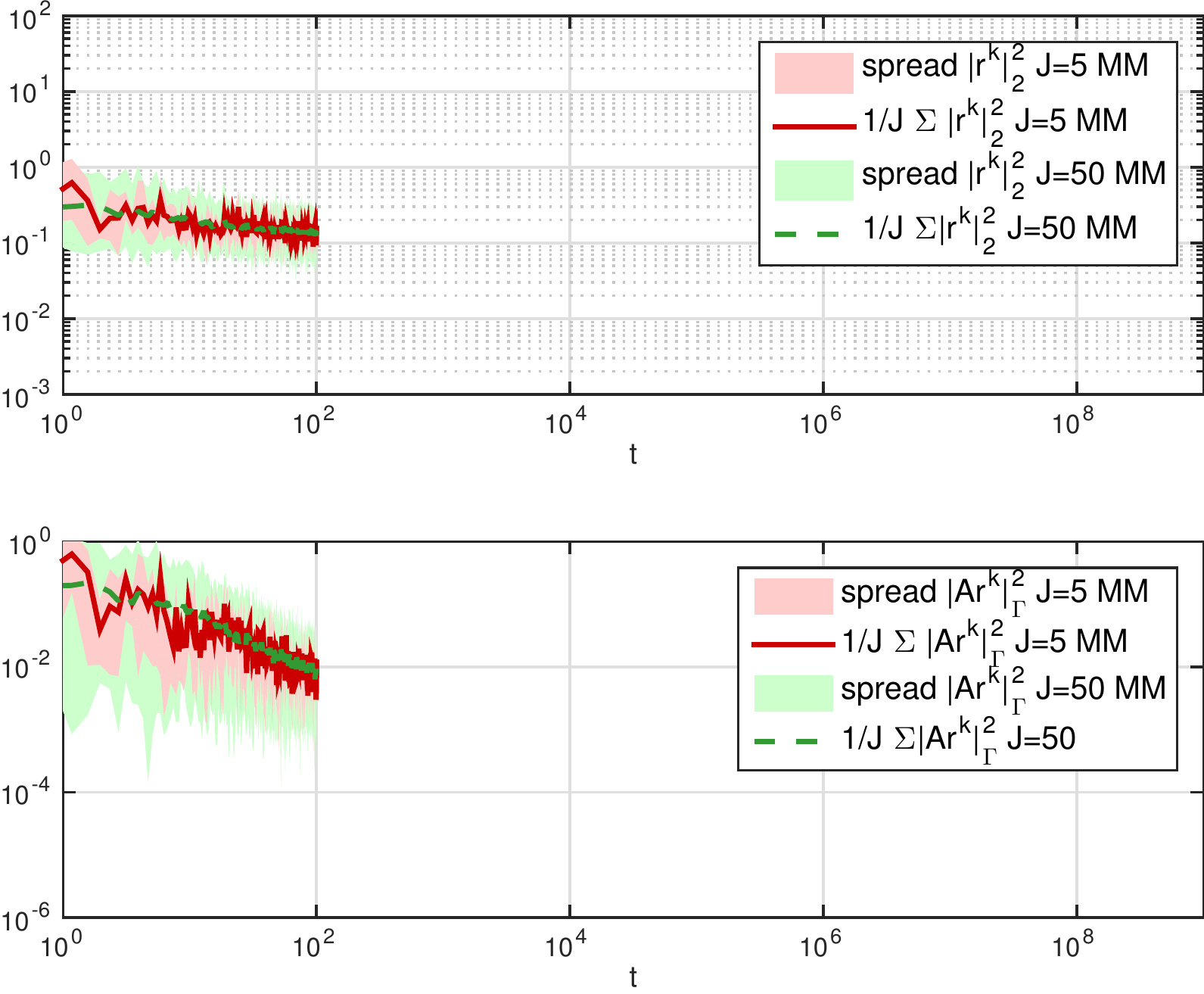}
~\\[-0.75cm]\caption{\footnotesize\label{fig:linNFrrsMM}
Quantities $|r|_2^2$, $|Ar|_{\Gamma}^2$ w.r. to time $t$, $J=5$ with randomized search (red) and $J=50$ with randomized search (green),  $\beta=10$, $K=2^4-1$, initial ensemble chosen based on KL expansion of $\cls{C_0}=\beta(A-id)^{-1}$. \color{white}{Comparison of the EnKF }\color{black}}
 \end{figure}
\end{minipage}
 \begin{minipage}{0.025\textwidth}
 ~
 \end{minipage}
\begin{minipage}{0.45\textwidth}
 \begin{figure}[H]
\centering
    \includegraphics[width=\textwidth]{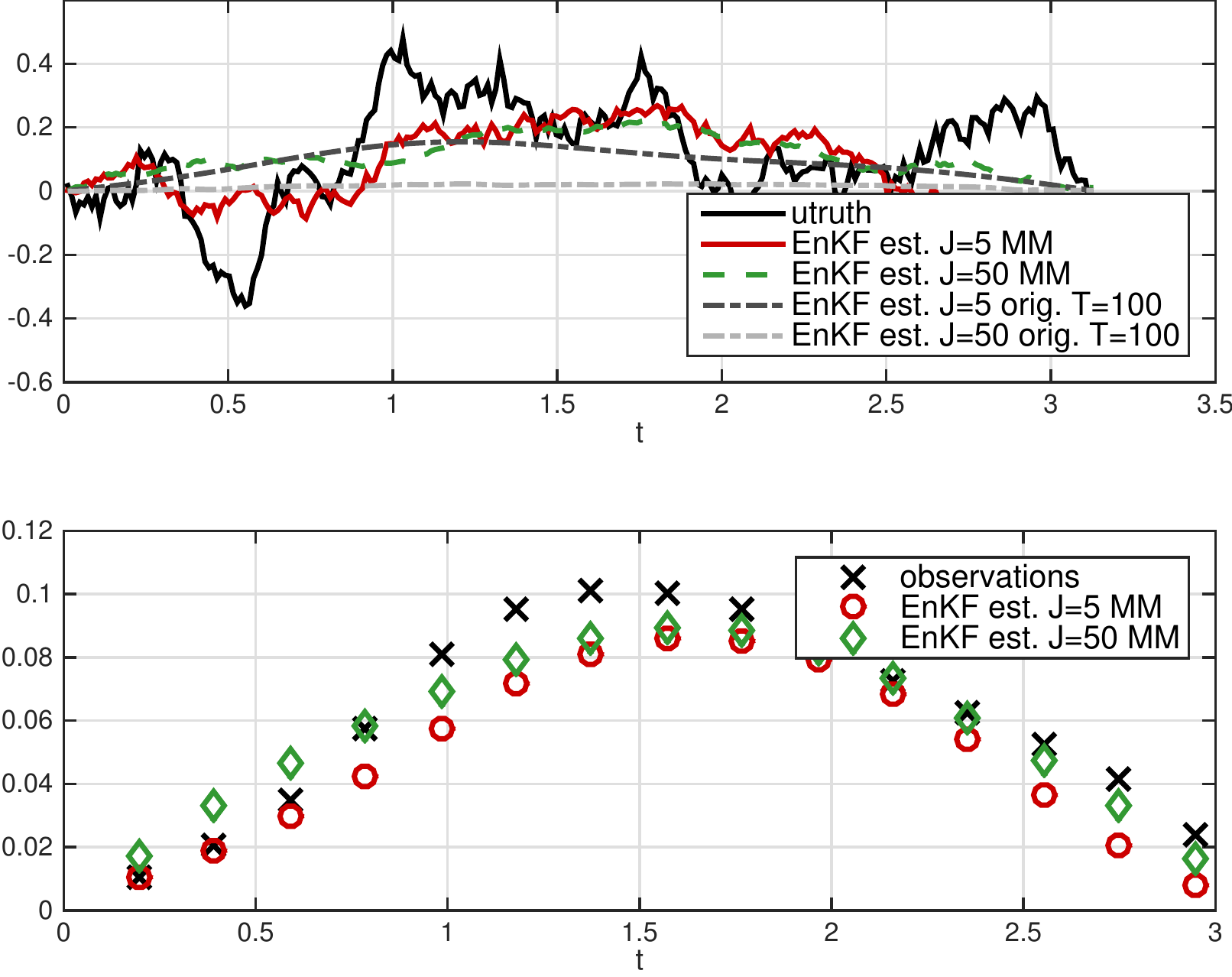}
~\\[-0.75cm]\caption{\footnotesize \label{fig:linNFsolrsMM}
Comparison of the EnKF estimate with the truth and the observations, $J=5$ with randomized search (red) and $J=50$ with randomized search (green), $\beta=10$, $K=2^4-1$, initial ensemble chosen based on KL expansion of $\cls{C_0}=\beta(A-id)^{-1}$. }
 \end{figure}
\end{minipage}
~\\[0.cm]

We study the same test case as in section \ref{sec:NEnoisy},
with the same realization of the measurement noise 
(cf. Figure \ref{fig:linNDrT1} and Figure \ref{fig:linNDsolT1}), to allow 
for a comparison of the three methods introduced in this
section. Combining the various techniques with the Bayesian stopping rule for noisy observations, we observe the following behavior given in Figure \ref{fig:linNFrrsnoise} and Figure \ref{fig:linNFsolrsnoise2}.

~\\[-0.3cm]
\begin{figure}[H]
\centering
    \includegraphics[width=0.75\textwidth]{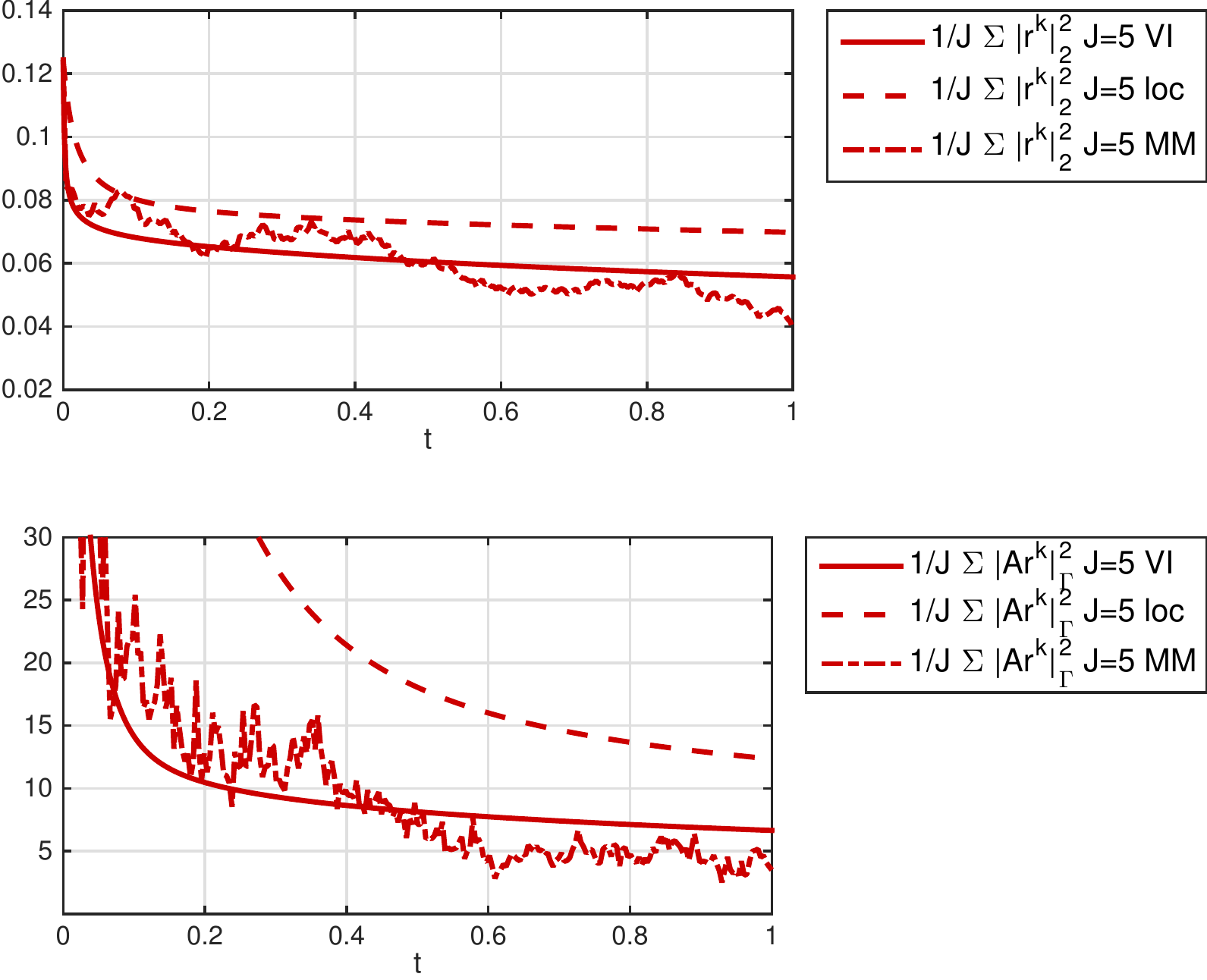}
~\\[-0.25cm]\caption{\footnotesize\label{fig:linNFrrsnoise}
Quantities $|r|_2^2$, $|Ar|_{\Gamma}^2$ w.r. to time $t$, $J=5$ (red) for the discussed variants, $\beta=10$, $\beta=10$, $K=2^4-1$, initial ensemble chosen based on KL expansion of $\cls{C_0}=\beta(A-id)^{-1}$. }
 \end{figure}
 In both figures, the abbreviation VI refers to variance inflation, loc denotes the localization technique and MM stands for the \cls{randomized search (Markov mixing)}.
The randomized search clearly outperforms the two other strategies and leads to a better estimate of the unknown data. \cls{In Figures \ref{fig:linNFrrsnoise} and \ref{fig:linNFsolrsnoise2}, one path of the solution of \eqref{eq:dl} is shown, similar performance can be observed for further paths.} This strategy has the potential to significantly improve the performance of the EnKF and will be investigated in more details in subsequent papers.

 \begin{figure}[H]
\centering
    \includegraphics[width=0.75\textwidth]{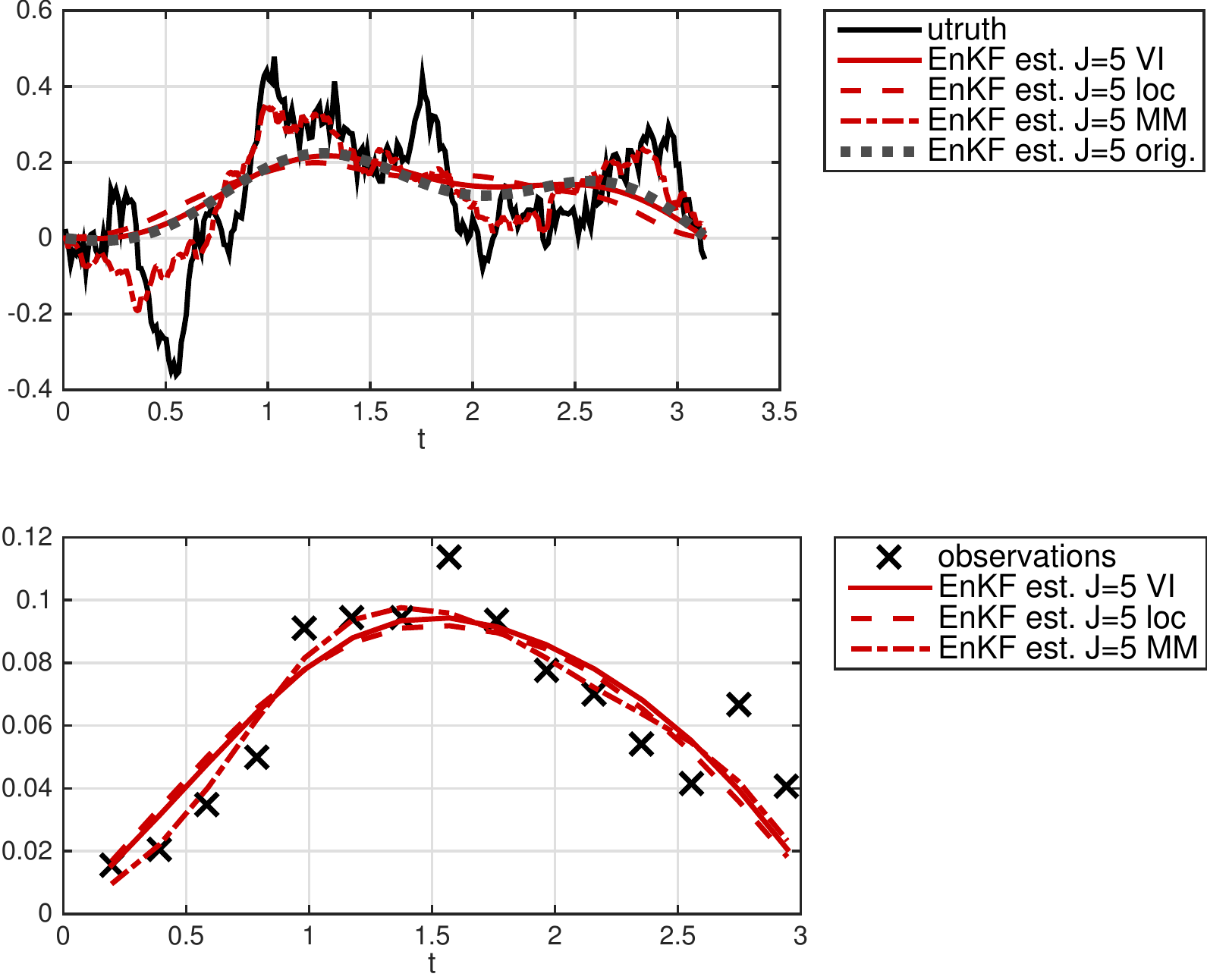}
~\\[-0.25cm]\caption{\footnotesize \label{fig:linNFsolrsnoise2}
Comparison of the EnKF estimate with the truth and the observations, $J=5$ (red) for the discussed variants, $\beta=10$, $K=2^4-1$, initial ensemble chosen based on KL expansion of $\cls{C_0}=\beta(A-id)^{-1}$. }
 \end{figure}
}

\section{Conclusions} Our analysis and numerical studies
for the ensemble Kalman filter applied to inverse problems demonstrate
several interesting properties: (i) the continuous time limit exhibits
structure that is hard to see in discrete time implementations used
in practice; (ii) in particular, for the linear inverse problem, it
reveals an underlying gradient flow structure; (iii) in the linear noise-free
case the method can be completely analyzed and this leads to a complete
understanding of error propagation; (iv) numerical results indicate that the conclusions
observed for linear problems carry over to nonlinear problems; (v) 
that the widely used localization and inflation techniques can improve
the method, but that the (introduced here for the first time) use of
ideas from SMC hold considerable promise for further improvement;
(vi) that importing stopping criteria and other regularization
techniques is crucial to the effectiveness of the method, as
highlighted by the work of Iglesias \cite{iglesias2014iterative,iglesias2015regularizing}. Our future work in this area, both theoretical and computational,
will reflect, and build on, these conclusions.

\vspace{0.2in}

\noindent{\bf Acknowledgments} Both authors are grateful to 
Dean Oliver for helpful advice, and to the 
EPSRC Programme Grant EQUIP for funding of this research. AMS is 
also grateful to DARPA and to ONR for funding parts of this research.\\[0.55cm]

%\section*{References}
\bibliographystyle{siamplain}

\bibliography{ref}

\section*{Appendix}\label{appendix}

\begin{lemma}
\label{lem:added}
The deviations from the mean $e^{(j)}$ and the deviations from
the truth $r^{(j)}$ satisfy
\begin{equation}
\label{eq:e}
\frac{\dd  e^{(j)}}{\dd t}= -\frac{1}{J}\sum_{k=1}^J E_{jk}e^{(k)}=-\frac{1}{J}\sum_{k=1}^J E_{jk}r^{(k)}
\end{equation}
and
\begin{equation}
\label{eq:r}
\frac{\dd  r^{(j)}}{\dd t}= -\frac{1}{J}\sum_{k=1}^J F_{jk}e^{(k)}=-\frac{1}{J}\sum_{k=1}^J F_{jk}r^{(k)}.
\end{equation}
\end{lemma}
\begin{proof}
Recall \eqref{eq:lode}: 
   \begin{equation}
   \frac{\dd u^{(j)}}{\dd t}=\frac{1}{J}\sum_{k=1}^J \bigl\langle A(u^{(k)}-\bu),
y-Au^{(j)}\bigr\rangle_{\Gamma} \bigl(u^{(k)}-\bu\bigr), \quad j=1,\cdots, J.
  \end{equation}
From this it follows that
\begin{equation}
\frac{\dd  \bu}{\dd t}= -\frac{1}{J}\sum_{k=1}^J \langle A(\bu-\ud), 
Ae^{(k)} \rangle{_\Gamma} e^{(k)}.
\end{equation}
Hence \eqref{eq:e} follows, with the second identity following from
the fact that $E\one=0.$ 
Since $\ud$ is time-independent we also have that \eqref{eq:r}
follows, with the second identity now following from
the fact that $F\one=0.$ 
\end{proof}

\begin{lemma} \label{l:EF}
Assume that $y$ is the image of a truth $\ud \in \cX$ under $A.$
The matrices $E$ and $F$ satisfy the equations
\begin{eqnarray*}
   \frac{\dd  }{\dd t}E=-\frac2JE^2, \quad \frac{\dd  }{\dd t}F=-\frac2JFE,
\quad \frac{\dd}{\dd t}R=-\frac2J FF^T.
     \end{eqnarray*}
As a consequence both $E$ and $F$ satisfy a global-in-time a
priori bound, depending only on initial conditions.
Explicitly we have the following. For the orthogonal matrix $X$ defined through
the eigendecomposition of $E(0)$ it follows that 
       \begin{eqnarray} \label{eq:diagE}
  E(t)=X\Lambda(t)X^{{\top}}\;
     \end{eqnarray}
     with
$\Lambda(t)=\mbox{diag}\{ \lambda^{(1)}(t),\ldots,\lambda^{(J)}(t)\}$,
$\Lambda(0)=\mbox{diag}\{ \lambda_0^{(1)},\ldots,\lambda_0^{(J)}\}$
and
            \begin{eqnarray} \label{eq:odeenslambda}
   \lambda^{(j)}( t)=\Big({\frac 2J t+\frac{1}{\lambda_0^{(j)}}}\Big)^{-1}\;,
     \end{eqnarray}
     if $ \lambda_0^{(j)}\neq 0$, otherwise $ \lambda^{(j)}(t)= 0$.
The matrix $R$ satisfies ${\rm Tr}\bigl(R(t)\bigr) \le {\rm Tr}\bigl(R(0)\bigr)$
for all $t \ge 0$, and 
$F_{ij} \rightarrow 0$ at least as fast as {$\frac{1}{\sqrt{t}}$} as $t \rightarrow \infty$ for each $i,j$ and, in particular, is bounded uniformly in time.
\end{lemma}

\begin{proof}
The first equation may be derived as follows:
     \begin{eqnarray*}
   \big(\frac{\dd  }{\dd t}E\big)_{ij}&=&\frac{\dd  }{\dd t} \langle Ae^{(i)},Ae^{(j)}\rangle_{\Gamma}\\
  &=&-\frac1J \sum_{k=1}^J E_{ik}\langle Ae^{(k)},Ae^{(j)}\rangle_{\Gamma}-\frac1J \sum_{k=1}^J E_{jk}\langle Ae^{(i)},Ae^{(k)}\rangle_{\Gamma}\\
  &=&-\frac2J \sum_{k=1}^J E_{ik}E_{kj}
   \end{eqnarray*}
as required.
The second equation follows similarly:
\begin{eqnarray}
   \big(\frac{\dd  }{\dd t}F\big)_{ij}&=&\langle A{\frac{\dd  }{\dd t}} r^{(i)},Ae^{(j)}\rangle_{\Gamma}+ \langle Ar^{(i)},A{\frac{\dd  }{\dd t}} e^{(j)}\rangle_{\Gamma}\nonumber\\
        &=&      -\frac1J \sum_{k=1}^J F_{ik}E_{kj}-\frac1J \sum_{k=1}^J F_{ik}E_{kj}\label{eq:Fglob}\;,
 \end{eqnarray}
as required; here we have used the fact that $F_{kj}-E_{kj}$ is independent of $k$
and hence, since $F\one=0$, 
$$\sum_{k=1}^J F_{ik}E_{kj}=\sum_{k=1}^J F_{ik}F_{kj}.$$

Due to the symmetry (and positive {semi}definiteness) of $E$, 
$E(0)$ is diagonalizable by orthogonal matrices, that is
$E(0)=X\Lambda(0)X^{{\top}}$, where 
$\Lambda(0)=\mbox{diag}\{ \lambda^{(1)}_0,\ldots,{\lambda^{(J)}_0}\}$. 
The solution of the ODE for $E(t)$ is therefore given by
       \begin{eqnarray} \label{eq:diagE2}
  E(t)=X\Lambda(t)X^{{\top}}\;
     \end{eqnarray}
     with $\Lambda(t)$ satisfying the following decoupled ODE
       \begin{eqnarray} \label{eq:odeenslambda2}
   \frac{\dd  \lambda^{(j)}}{\dd t}=-\frac2J(\lambda^{(j)})^2\;.
     \end{eqnarray}
     The solution of \eqref{eq:odeenslambda2} is thus given by
            \begin{eqnarray} \label{eq:odeenslambda3}
   \lambda^{(j)}( t)=\big({\frac 2J t+\frac{1}{\lambda^{(j)}_0}}\big)^{-1}\;,
     \end{eqnarray}
     if $ \lambda^{(j)}_0\neq 0$, otherwise $ \lambda^{(j)}(t)= 0$.
     The behavior of $R$ is described by
   \begin{eqnarray*}
   \big(\frac{\dd  }{\dd t}R\big)_{ij}&=&\langle A\frac{\dd  }{\dd t}  r^{(i)},Ar^{(j)}\rangle_{\Gamma}+ \langle Ar^{(i)},A\frac{\dd  }{\dd t}  r^{(j)}\rangle_{\Gamma}\\
        &=&      -\frac1J \sum_{k=1}^J F_{ik}F_{jk}-\frac1J \sum_{k=1}^J F_{jk}F_{ik}\;,
         \end{eqnarray*}
     and thus
       \begin{equation}
        \frac{\dd  }{\dd t}R=-\frac2JFF^\top.
       \end{equation}
Taking the trace of this identity gives
\begin{equation}
        \frac{\dd  }{\dd t}{\rm Tr}(R)=-\frac2J\|F\|^2_{{\rm Fr}},
       \end{equation}
where $\|\cdot\|_{{\rm Fr}}$ is the Frobenius norm.
The bound on the trace of $R$ follows. 

By the Cauchy-Schwartz inequality, we have
 \begin{eqnarray*}
 F_{ij}^2=\langle Ar^{(i)},Ae^{(j)}\rangle_{\Gamma}^2\le |Ar^{(i)}|_{\Gamma}^2\cdot |Ae^{(j)}|_{\Gamma}^2 \le C |Ae^{(j)}|_{\Gamma}^2\;,
 \end{eqnarray*}
 and hence, $F_{ij} \rightarrow 0$ at least as fast as {$\frac{1}{\sqrt{t}}$} as $t \rightarrow \infty$ as required.
\end{proof}
\begin{lemma}
\label{l:er}

Assume that $y$ is the image of a truth $\ud \in \cX$ under $A$ and the forward operator $A$ is one-to-one.
Then 
\begin{subequations}
\label{eq:er}
\begin{align}
Ae^{(j)}(t)&=\sum_{k=1}^{J} \ell_{jk}(t)Ae^{(k)}(0),\\
Ar^{(j)}(t)&=\sum_{k=1}^{J} q_{jk}(t)Ae^{(k)}(0)+\rho^{(j)}(t)
\end{align}
\end{subequations}
where the matrices $L=\{\ell_{jk}\}$ and $Q=\{q_{jk}\}$ 
satisfy
\begin{subequations}
\label{eq:LM}
\begin{align}
\frac{\dd L}{\dd t}=-\frac{1}{J}EL, \qquad
\frac{\dd Q}{\dd t}=-\frac{1}{J}FL,
\end{align}
\end{subequations}
and $\rho^{(j)}(t)=\rho^{(j)}(0)=\rho^{(1)}(0)$ is the projection of
$Ar^{(j)}(0)$ into the subspace which is orthogonal in $\cY$ to the linear span
of $\{Ae^{(k)}(0)\}_{k=1}^J$, with respect to the inner product $\langle\cdot,\cdot\rangle_\Gamma$.
As a consequence 
       \begin{eqnarray} \label{eq:diagE3}
  L(t)=X\Omega(t)X^{{\top}}\;
     \end{eqnarray}
     with
$\Omega(t)=\mbox{diag}\{ \omega^{(1)}(t),\ldots,\omega^{(J)}(t)\}$,
$\Omega(0)=I$
and
            \begin{eqnarray} \label{eq:odeenslambda4}
   \omega^{(j)}( t)={\Big(\frac 2J \lambda_0^{(j)}t+1\Big)}^
{-\frac12}\;.
     \end{eqnarray}
We also assume that the rank of the subspace spanned by the
vectors $\{Ae^{(j)}(t)\}_{j=1}^J$ is equal to $\tilde J$ and that 
(after possibly reordering the eigenvalues) 
$\lambda^{(1)}(t)=\ldots=\lambda^{(J-\tilde J)}(t)=0$ and 
$\lambda^{(J-\tilde J+1)}(t),\ldots,\lambda^{(J)}(t)>0$. It then
follows that $\omega^{(1)}(t)=\ldots=\omega^{(J-\tilde J)}(t)=1$.  
Furthermore, $L(t)x^{(k)}=L(t)^{-1}x^{(k)}=x^{(k)}$ for all $t \ge 0,\ k=1,\dots,J-\tilde J$, where $x^{(k)}$ are the columns of $X$.
Without loss of generality we may assume that {$Q(t)x^{(k)}=0$ for all $t \ge 0, \ k=1,\dots,J-\tilde J.$}

\end{lemma}

\begin{proof} 
Differentiating expression (\ref{eq:er}a) and substituting in \eqref{eq:e}
from Lemma \ref{lem:added} gives 
$$\sum_{m=1}^J \frac{\dd \ell_{jm}}{\dd t}{A}e^{(m)}(0)=-\frac{1}{J}
\sum_{k=1}^J \sum_{m=1}^J E_{jk} \ell_{km}{A}e^{(m)}(0).$$
Reordering the double summation on the right-hand side
and re-arranging we obtain
$$\sum_{m=1}^J\Bigl(\frac{\dd \ell_{jm}}{\dd t}+\frac{1}{J}
\sum_{k=1}^JE_{jk} \ell_{km}\Bigr){A}e^{(m)}(0)=0.$$
This is satisfied identically if equation (\ref{eq:LM}a) holds.
By uniqueness choosing the ${A}e^{(j)}(t)$ to be defined in this way gives
the unique solution for their time evolution.

Now we differentiate expression (\ref{eq:er}b) and substitute into
\eqref{eq:r} from Lemma \ref{lem:added}.
A similar analysis to the preceding yields
$$\sum_{m=1}^J\Bigl(\frac{\dd Q_{jm}}{\dd t}+\frac{1}{J}
\sum_{k=1}^JF_{jk} \ell_{km}\Bigr){A}e^{(m)}(0)+\frac{\dd \rho^{(j)}}{\dd t}=0.$$
Again this can be satisfied identically if equation (\ref{eq:LM}b) holds
and if $\rho^{(j)}(t)$ is the constant function as specified above. 
By uniqueness we have the desired solution. The independence of  $\rho^{(j)}(t)$ with respect to $j$, i.e. $\rho^{(j)}(0)=\rho^{(1)}(0)$, follows from the fact
that $\rho^{(j)}(0)$ is the function inside the norm $\|\cdot\|_{\Gamma}$
which  is found by choosing the vector $q_j:=\{q_{jk}\}_{k=1}^J$ so as
to minimize the functional 
$$\|Ar^{(j)}(0)-\sum_{k=1}^{J} q_{jk}(0)Ae^{(k)}(0)\|_{\Gamma}.$$
From the definition of the $r^{(j)}$ and $e^{(j)}$ this is equivalent to
determining the function inside the norm $\|\cdot\|_{\Gamma}$
found by choosing the vector $q_j:=\{q_{jk}\}_{k=1}^J$ so as
to minimize the functional
$$\|Au^{(j)}(0)-\sum_{k=1}^{J} q_{jk}(0)A(u^{(k)}(0)-\bar u(0))-A\ud\|_{\Gamma}.$$
This in turn is equivalent to
determining the function inside the norm $\|\cdot\|_{\Gamma}$
found by choosing the vector $\tilde q:=\{\tilde q_{k}\}_{k=1}^J$ so as
to minimize the functional
$$\|\sum_{k=1}^{J} \tilde q_{k}Au^{(k)}(0)-A\ud\|_{\Gamma}$$
and is hence independent of $j$.
Our assumptions on the span of $\{Ae^{(j)}(t)\}_{j=1}^J$ imply that
$E(0)$ has exactly $J-\tilde{J}$ zero eigenvalues, corresponding to eigenvectors
$\{x^{(k)}\}_{k=1}^{J-\tilde{J}}$ with the property that
$$\sum_{j=1}^J x^{(k)}_j Ae^{(j)}(0)=0.$$
(One of these vectors $x^{(k)}$ is of course $\one$ so that $\tilde{J} \ge 1.$) 
As a consequence we also have
\begin{equation}
\label{eq:dda}
E(0)x^{(k)}=F(0)x^{(k)}=0\quad k=1, \cdots, J-\tilde{J}.
\end{equation}
The fact that $L(t)x^{(k)}=x^{(k)}$ for all $t \ge 0, \ k=1,\ldots, \tilde J$
is immediate from the fact that $L=X\Omega X^{\top}$, because
$x^{(k)}$ is the eigenvector corresponding to eigenvalue $\omega^{(k)}(t)=1\ k=1,\dots,\tilde J$; an identical argument shows the same for $L(t)^{-1}.$
The property that $Q(t)x^{(k)}=0$ for all $t \ge 0, \ k=1,\ldots, \tilde J$ follows by choosing
$Q(0)$ so that $Q(0)x^{(k)}=0$, which is always possible because the
$x^{(k)}$ are eigenvectors with corresponding eigenvalues $\lambda^{(k)}=0$, and then noting
that $Q(t)x^{(k)}=0$ for all time because $F(t)L(t)x^{(k)}=F(t)x^{(k)}=0$ for all $t
\ge 0.$
The last item is zero because $Ex^{(k)}=0$ and because 
$\frac{\dd  }{\dd t}F=-\frac2JFE$; we also use that $F(0)x^{(k)}=0$
from \eqref{eq:dda}.

\end{proof}

\end{document}